\documentclass[12pt]{article}

\usepackage{amsmath,amsthm,amsfonts}
\usepackage{bm}
\usepackage{mathtools}
\usepackage{enumitem}
\usepackage{hyperref}
\usepackage{natbib}
\usepackage{float}
\usepackage{geometry}

\makeatletter
\@tfor\next:=abcdefghijklmnopqrstuvwxyzABCDEFGHIJKLMNOPQRSTUVWXYZ\do{%
  \def\command@factory#1{%
    \expandafter\def\csname B#1\endcsname{\mathbf{#1}}
  }
 \expandafter\command@factory\next
}

\@tfor\next:=ABCDEFGHIJKLMNOPQRSTUVWXYZ\do{%
  \def\command@factory#1{%
    \expandafter\def\csname mb#1\endcsname{\mathbb{#1}}
  }
 \expandafter\command@factory\next
}
\@tfor\next:=ABCDEFGHIJKLMNOPQRSTUVWXYZ\do{%
  \def\command@factory#1{%
    \expandafter\def\csname mc#1\endcsname{\mathcal{#1}}
  }
 \expandafter\command@factory\next
}
\makeatother

\newtheorem{theorem}{Theorem}
\newtheorem{proposition}{Proposition}
\newtheorem{corollary}{Corollary}
\newtheorem{lemma}{Lemma}

\theoremstyle{definition}
\newtheorem{remark}{Remark}

\DeclareMathOperator*{\argmax}{argmax}
\DeclarePairedDelimiter\abs{\lvert}{\rvert}
\DeclarePairedDelimiter\norm{\lVert}{\rVert}
\newcommand{\cbr}[1]{\left(#1\right)}
\newcommand{\sbr}[1]{\left[#1\right]}

\newcommand{\interior}[1]{{\kern0pt#1}^{\mathrm{o}}}

\title{Asymptotically well-calibrated Bayesian $p$-value using the Kolmogorov-Smirnov statistic}
\author{Yueming Shen, Surya T. Tokdar \\ Department of Statistical Science \\ Duke University}
\date{\today}

\begin{document}
\maketitle

\begin{abstract}
    The posterior predictive $p$-value (ppp) is widely used in Bayesian model evaluation. However, due to double use of the data, the ppp may not be a valid $p$-value even in large samples: The asymptotic null distribution of the ppp can be non-uniform unless the underlying test statistic satisfies certain well-calibration conditions. Such conditions have been studied in the literature for asymptotically normal test statistics. We extend this line of work by establishing well-calibration conditions for test statistics that are not necessarily asymptotically normal. In particular, we show that Kolmogorov-Smirnov (KS)-type test statistics satisfy these conditions, such that their ppps are asymptotically well-calibrated Bayesian $p$-values. KS-type statistics are versatile, omnibus, and sensitive to model misspecifications. They apply to i.i.d. real-valued data, as well as non-identically distributed observations under regression models. Numerical experiments demonstrate that such $p$-values are well behaved in finite samples and can effectively detect a wide range of alternative models.

\smallskip


\end{abstract}
\section{Introduction}

For Bayesians, assessing the adequacy of a fitted model is no simple task. In principle, at the end of the analysis pipeline, no obvious alternative model is left to explore once a ``best'' model has been chosen during intermediate stages of model selection and comparison. Arguably, all that remains to critique is any lingering secondary disagreements between the observed data and the hypothetical data implied by the fitted model. This perspective is clearly Fisherian in its framing and could be approached statistically through Fisher's significance testing. Indeed, one of the most well known tools for Bayesian model checking is the posterior predictive $p$-value \citep[ppp,][]{guttman1967goodnessoffit,rubin1984bayesianly,gelman1996posterior} which tries to measure how much ``out in the tails'' the observed data is relative to synthetic data one would generate from the fitted posterior predictive distribution. 

To be precise, suppose we posit a parametric sampling model $\mcP=\{P_{\bm{\theta}}: \bm{\theta} \in \Theta \subset \mbR^k \}$ for observed data of length $n$, $\By_n \in \mcY_n$, along with a prior $\pi(\bm{\theta})$ on the parameter $\bm{\theta}$. In using the ppp, one carries out significance testing of the composite hypothesis $H: \BY_n \sim P_{\bm{\theta}}$, $\bm{\theta} \in \Theta$. With a slight abuse of notation, we use $P_{\bm{\theta}}$ to denote both the probability measure and its cumulative distribution function (CDF). Let $t(\BY_n): \mcY_n \to \mbR$ be a test statistic whose larger values indicate greater discrepancy between observed data $\By_n$ and the hypothesis $H$, $p_{\bm{\theta}}(\By)$ denote the density of $P_{\bm{\theta}}$ under the Lebesgue measure, and $\pi(\bm{\theta} \mid \By_n)$ denote the posterior distribution. The ppp is defined as:
\begin{align*}
  \mathrm{ppp}(\By_n) = P^{m(\By \mid \By_n)}[t(\BY_n) \ge t(\By_n)],
\end{align*}
where the reference distribution for $\BY_n$ is the posterior predictive distribution
\begin{align*}
  m(\By \mid \By_n) = \int_{\Theta}p_{\bm{\theta}}(\By)\pi(d\bm{\theta} \mid \By_n).
\end{align*}
Smaller $p$-values provide stronger evidence against $H$. 

The ppp is well known, intuitive, and easy to implement. And yet, there are challenges to its practical use. The validity of Fisher's significance testing critically depends on the $p$-value distribution being uniform under the null, at least approximately if not exactly. However, the ppp often fails to be uniformly distributed largely due to the ``double use'' of data: the same data is used to fit the posterior and assess disagreement. \cite{robins2000asymptotic} show that even in the limit as sample size grows to infinity, the ppp can be too conservative. The probability that the ppp falls below a small number $\alpha$ can be much smaller than $\alpha$ itself. Ideally, this probability should be exactly $\alpha$, so that reporting a $p$-value of $\alpha$ conveys the message that there is exactly $\alpha$ probability of observing more extreme data under $H$. $P$-values that are uniformly distributed under the null satisfy this requirement and are called {\it well calibrated}.

Well-calibration is fundamentally a frequentist concept. Bayesians with subjective prior information might not be concerned with well-calibration under all parameter values under a composite null. However, in the absence of strong prior information, Bayesians often adopt reference priors to maximize learning from the data, or use weak, convenient priors for practicality. In such cases where the prior may not encode strong subjective beliefs but primarily facilitates analysis, well-calibration could be important to Bayesians as it ensures reliability of the $p$-value for model criticism. This is the setting we focus on in this article. 

One way to make the ppp approximately well calibrated is to choose the test statistic $t(y)$ cleverly. \cite{robins2000asymptotic} demonstrate that the ppp is asymptotically well calibrated if $t(y)$ is asymptotically normal with an asymptotic mean that does not depend on the unknown parameter $\bm\theta$. However, this requirement severely restricts the choice of test statistics one could consider for a given analysis. Statistics satisfying this condition may not be sensitive to mismatch between $y$ and $H$, whereas many practically useful statistics do not satisfy it.

These challenges lead to the question: are omnibus, non-bespoke constructions of $t(y)$ available for which the ppp is well calibrated? We show here that for real-valued random samples under parametric models, such a construction is available in the form of \textit{Kolmogorov-Smirnov} (KS)-type statistics. These are omnibus test statistics which provide a global nonparametric assessment of data-model compatibility, well suited for catching secondary deficiencies of a parametric model, e.g., lack of skewness, kurtosis, or bimodality of a normal model. They are widely applicable, sensitive to model misspecification, and crucially, as we will demonstrate, yield asymptotically well-calibrated ppp without satisfying the conditions of \cite{robins2000asymptotic}. 

Let $Y_1,\dots,Y_n$ be real-valued random variables. Consider testing for $H: Y_i \sim$ i.i.d. $P_{\bm{\theta}}$. The KS test quantifies discrepancy between the empirical distribution of the sample and $P_{\bm{\theta}}$. It is defined as $\sqrt{n}\sup_{y \in \mbR} \abs{P_n(y)-P_{\bm{\theta}}(y)}$, where $P_n$ denotes the empirical CDF. Replacing $\bm{\theta}$ with an estimator $\hat{\bm{\theta}}_n$ yields the {\it classical modified KS} (CKS) test statistic 
\begin{align*}
  K_n=\sqrt{n}\sup_{y \in \mbR}\abs{P_n(y) - P_{\hat{\bm{\theta}}_n}(y)}.
\end{align*}
The scaling by $\sqrt{n}$ ensures that under proper assumptions, $K_n$ converges in distribution to a continuous random variable, rather than a point mass at zero. For more complex settings such as regression models where the observations are independent but not identically distributed, let $P_{i,\bm{\theta}}$ denote the CDF of $Y_i$, a {\it generalized modified KS} (GKS) test statistic $\tilde{K}_n$ can be used:
\begin{align*}
  \tilde{K}_n = \sqrt{n} \sup_{u \in [0,1]} \abs{u - F_{n,\hat{\bm{\theta}}_n}(u)}, \quad F_{n,\hat{\bm{\theta}}_n}(u) = \frac 1n \sum_{i=1}^n 1(P_{i,\hat{\bm{\theta}}_n}(Y_i) \le u).
\end{align*}
This test statistic is reasonable because under the null, $P_{i,\bm{\theta}}(Y_i)$'s are i.i.d. uniform.

Asymptotic properties of the ppp under the CKS and GKS statistics (denoted as ppp(CKS) and ppp(GKS)) remain unexplored in the literature. Simply knowing the asymptotic distributions of $K_n$ and $\tilde{K}_n$ is not enough to establish the limiting distributional properties of the ppp, which additionally involves the posterior predictive distribution under a given prior. As shown in \cite{robins2000asymptotic}, the asymptotic behavior of the ppp depends on the limiting distribution of the test statistic under ``local alternatives'', as well as concentration properties of the posterior distribution. Unfortunately, the theory of \cite{robins2000asymptotic} only applies to asymptotically normal test statistics, whereas, both $K_n$ and $\tilde{K}_n$ converge in distribution to the supremum of a mean-zero gaussian process (\cite{van2000asymptotic}) under the null, and hence are not asymptotically normal. Therefore, new theoretical tools are needed to study the asymptotic behavior of the ppp(CKS) and ppp(GKS).

In Section \ref{sec:main}, we establish a general result that a sufficient condition for the ppp to be asymptotically well calibrated is to have the test statistic converge in distribution to a continuous random variable under contiguous alternatives $P_{\bm{\theta}_n}$, for any $\bm{\theta}_n$ such that $\norm{\bm{\theta}_n-\bm{\theta}_0}= O(1/\sqrt{n})$. We then show that both the CKS and GKS tests satisfy this condition. For the CKS test, we first apply Le Cam's third lemma to verify the condition for $\sqrt{n}(P_n(y)-P_{\hat{\bm{\theta}}_n}(y))$ for any single $y$. Using empirical process theory, we extend this result uniformly in $y$ to establish the condition for $K_n$. The argument for the GKS test follows a similar structure but is technically more involved because it requires handling non-continuous transformations of the estimators. We additionally use symmetrization, Dudley’s entropy integral theorem, and VC-dimension bounds derived from o-minimal structures to establish the result for $\tilde{K}_n$. In Section \ref{sec:sim}, we study finite-sample behavior of the ppp(CKS) and ppp(GKS) through numerical experiments, showing that both are well behaved under the null and exhibit strong power against a range of alternatives. Then we conclude with discussions in Section \ref{sec:dis}. Our results show that well-calibrated ppp-based Bayesian model checks could be constructed outside of the theory of \cite{robins2000asymptotic}, with the ppp(CKS) and ppp(GKS) being compelling canonical choices.

\section{Asymptotic well-calibration}\label{sec:main}

\subsection{General framework}

We first establish sufficient conditions on the asymptotic behavior of a test statistic for the ppp to be asymptotically well calibrated. Let $\BY_n \in \mbR^n$ denote the random data vector of length $n$, and $\By_n \in \mbR^n$ denote a realization of $\BY_n$. The data are modeled as $\BY_n \sim P_{\bm{\theta}}^n$, $\bm{\theta} \subset \Theta \subset \mbR^k$. Let $\bm{\theta}_0 \in \interior{\Theta}$ denote the unknown true parameter, where $\interior{\Theta}$ is the interior of $\Theta$.

\begin{proposition}\label{prop:paper5thm3.2}
  
  For any constant $c \in \mbR$, define $A_n(c) := \{\bm{\theta} \in \Theta, \norm{\bm{\theta}-\bm{\theta}_0} \le c/\sqrt{n} \}$. Let $T_n$ be a test statistic with CDF $G_n(t \mid \bm{\theta})$. If 
  \begin{enumerate}[label=P\arabic{*}., ref=P\arabic{*}]
      \item \label{assmp:pc} for any $\epsilon>0$, there exists $c_\epsilon$, $N_\epsilon>0$ such that for all $n>N_\epsilon$, 
      \begin{align*}
          P_{\bm{\theta}_0}^n(\pi(A_n(c_\epsilon) \mid \BY_n) \ge 1-\epsilon) \ge 1-\epsilon;
      \end{align*} 
      \item \label{assmp:cont} $T_n \rightsquigarrow T$ for some random variable $T$ with continuous CDF;
      \item \label{assmp:uc} for any $c \in \mbR$, $\sup_{\bm{\theta} \in A_n(c),t \in \mbR}\abs{G_n(t \mid \bm{\theta}) - G_n(t \mid \bm{\theta}_0)} \to 0$. 
  \end{enumerate}
  Then the ppp of $T_n$ is asymptotically well calibrated.
\end{proposition}

\begin{proof}
    This proposition makes reference to Theorem 3.2 in \cite{wang2021calibration}. Let $p_0(\By_n)=1-G_n(t_n \mid \bm{\theta}_0)$ denote the ppp for testing $H_0:\bm{\theta}=\bm{\theta}_0$. Observe that
    \begin{align*}
        &\sup_{\alpha \in [0,1]} \abs*{P_{\bm{\theta}_0}^n(\mathrm{ppp}(\BY_n) \le \alpha) - \alpha} \\
        &\le \sup_{\alpha \in [0,1]} \abs*{P_{\bm{\theta}_0}^n(p_0(\BY_n) \le \alpha) - \alpha} + \sup_{\alpha \in [0,1]} \abs*{P_{\bm{\theta}_0}^n(\mathrm{ppp}(\BY_n) \le \alpha) - P_{\bm{\theta}_0}^n(p_0(\BY_n) \le \alpha)}.
    \end{align*}
    Therefore the proof is complete once we can show these two terms converge to zero.

    Notice that for any $\alpha \in [0,1]$, $P_{\bm{\theta}_0}(p_0(\BY_n) \le \alpha)= 1-G_n(G_n^-(1-\alpha \mid \bm{\theta}_0) \mid \bm{\theta}_0)$, where $G_n^-(\alpha \mid \bm{\theta}):=\inf\{t \in \mbR, G_n(t \mid \bm{\theta}) \ge \alpha\}$ is the pseudo-inverse of $G_n$. By Lemma 21.2 in \cite{van2000asymptotic}, \ref{assmp:cont} implies that $G_n^-$ converges pointwise to $G^-$, hence $G_n^-(1-\alpha \mid \bm{\theta}_0) \to G^-(1-\alpha \mid \bm{\theta}_0)$. Because $G$ is continuous, $G_n$ converges uniformly to $G$ by Lemma \ref{lem:cdf_conv}. Thus by Lemma \ref{lem:uc}, $G_n(G_n^-(1-\alpha)) \to G(G^-(1-\alpha)) = 1-\alpha$, which shows that $p_0(\BY_n)$ converges in distribution to uniform. Lemma \ref{lem:cdf_conv} then implies the first term converges to zero.

    To show the second term converges to zero, it suffices to show $E_{\bm{\theta}_0} \abs{\mathrm{ppp}(\BY_n) - p_0(\BY_n)} \to 0$. For any $c \in \mbR$,
    \begin{align*}
      &E_{\bm{\theta}_0} \abs*{\mathrm{ppp}(\BY_n) - p_0(\BY_n)} = E_{\bm{\theta}_0}\abs*{\int_{\Theta} [G_n(T_n \mid \bm{\theta}_0) - G_n(T_n \mid \bm{\theta})] \pi(d\bm{\theta} \mid \BY_n)} \\
      \le& E_{\bm{\theta}_0}\abs*{\int_{\Theta \cap A_n(c)}[G_n(T_n \mid \bm{\theta}_0) - G_n(T_n \mid \bm{\theta})] \pi(d\bm{\theta} \mid \BY_n)} \\
      &\qquad{}+ E_{\bm{\theta}_0}\abs*{\int_{\Theta \backslash A_n(c)}[G_n(T_n \mid \bm{\theta}_0) - G_n(T_n \mid \bm{\theta})] \pi(d\bm{\theta} \mid \BY_n)} \\
      \le& \underbrace{\sup_{\bm{\theta} \in A_n(c)} \sup_{t \in \mbR} \abs{G_n(t \mid \bm{\theta}_0) - G_n(t \mid \bm{\theta})}}_{\text{(a)}} + \underbrace{1- E_{\bm{\theta}_0} \pi(A_n(c) \mid \BY_n)}_{\text{(b)}}.
    \end{align*}
    By \ref{assmp:pc}, for any $\epsilon >0$, there exists $c_\epsilon$, and $N_\epsilon$ large such that for all $n>N_\epsilon$,
    \begin{align*}
        P_{\bm{\theta}_0}^n(\pi(A_n(c_\epsilon) \mid \BY_n) \ge 1-\epsilon) \ge 1-\epsilon.
    \end{align*}
    Therefore $E_{\bm{\theta}_0} [\pi(A_n(c) \mid \BY_n)] \ge (1-\epsilon)^2$, and (b) $\le 1-(1-\epsilon)^2 < 2\epsilon$. By \ref{assmp:uc}, with $c_\epsilon$ chosen as above, term (a) $\to 0$ as $n$ increases. As $\epsilon$ is arbitrary, (a)+(b) $\to 0$. This concludes the proof.
\end{proof}

Next, we present slightly stronger but more easily verifiable conditions that lead to assumptions \ref{assmp:pc}-\ref{assmp:uc}. Assumption \ref{assmp:pc} can be established under the same regularity conditions that guarantee the Bernstein-von Mises (BvM) property of a prior for parametric models. Let $d_{TV}(P,Q)=\sup_{B \in \mcB}\abs{P(B)-Q(B)}$ denote the total variation distance between probability measures $P$ and $Q$ defined on some probability space $(\Omega,\mcB)$. 

\begin{lemma}\label{lem:l1}
    If there exists a consistent and asymptotically normal estimator $\tilde{\bm{\theta}}_n$ such that $\sqrt{n}(\tilde{\bm{\theta}}_n-\bm{\theta}_0) \rightsquigarrow N(\bm{0},I(\bm{\theta}_0)^{-1})$, for some positive definite matrix $I(\bm{\theta}_0)$, and the posterior has the BvM property: $d_{TV}(\pi(\cdot \mid \BY_n),\hat{\pi}_n) \overset{p}{\to} 0$, where $\hat{\pi}_n$ is the probability measure for $N(\tilde{\bm{\theta}}_n,I(\bm{\theta}_0)^{-1}/n)$, then \ref{assmp:pc} holds.
\end{lemma}

\begin{proof}[Proof of Lemma \ref{lem:l1}]

    Recall that $A_n(c)=B(c/\sqrt{n},\bm{\theta}_0) \cap \Theta$, where $B(r,\Ba)$ denotes an open ball of radius $r$ centered at $\Ba$. By asymptotic normality, for any $\epsilon>0$, there exists $c_\epsilon^* >0$ and $N_\epsilon^*$ such that $P_{\bm{\theta}_0}^n (\tilde{\bm{\theta}}_n(\BY_n) \in A_n(c_\epsilon^*)) \ge 1-\epsilon/2$ for all $n > N_\epsilon^*$. Find $d_\epsilon>0$ such that the $N(\bm{0},I(\bm{\theta}_0)^{-1})$ distribution places at least $1-\epsilon/2$ probability in $B(d_\epsilon,\bm{0}) \subset \mbR^k$. Clearly, $\hat{\pi}_n(B(d_\epsilon/\sqrt{n},\tilde{\bm{\theta}}_n)) \ge 1-\epsilon/2$ under $P_{\bm{\theta}_0}^n$. Set $c_\epsilon := c_\epsilon^* + d_\epsilon$. Notice that $\tilde{\bm{\theta}}_n \in A_n(c_\epsilon^*)$ implies $B(d_\epsilon/\sqrt{n},\tilde{\bm{\theta}}_n) \subset A_n(c_\epsilon)$. Therefore
    \begin{align*}
        P_{\bm{\theta}_0}^n \cbr{\hat{\pi}_n(A_n(c_\epsilon)) \ge 1-\frac{\epsilon}2} \ge P_{\bm{\theta}_0}^n(\tilde{\bm{\theta}}_n \in A_n(c_\epsilon^*)) \ge 1-\frac{\epsilon}2,
    \end{align*}
    for all $n > N_\epsilon^*$. From here onwards we use $A_n$ in place of $A_n(c_\epsilon)$ for simplicity. Let $\pi_n$ denote the posterior measure $\pi(\cdot \mid \BY_n)$. Then
    \begin{align*}
        P_{\bm{\theta}_0}^n(\pi_n(A_n) < 1-\epsilon) \le P_{\bm{\theta}_0}^n\cbr{\hat{\pi}_n(A_n) < 1- \frac \epsilon2} + P_{\bm{\theta}_0}^n\cbr{\abs{\hat{\pi}_n(A_n)-\pi_n(A_n)} \ge \frac{\epsilon}2}.
    \end{align*}
    The first term on the right is bounded by $\epsilon/2$ for all $n>N_\epsilon^*$. By the BvM property, $\abs{\hat{\pi}_n(A_n)-\pi_n(A_n)} \overset{p}{\to}0$, therefore there exists $N_\epsilon>N_\epsilon^*$ such that the second term is smaller than $\epsilon/2$ for all $n > N_\epsilon$. This concludes the proof.

\end{proof}

\begin{lemma}\label{lem:l2l3}
  A sufficient condition for \ref{assmp:cont} and \ref{assmp:uc} is
  \begin{align}\label{eq:thm_iid_c_step1}
    T_n \overset{P_{\bm{\theta}_n}^n}{\rightsquigarrow} T, \quad \text{ for every sequence } \bm{\theta}_n \text{ with } \norm{\bm{\theta}_n-\bm{\theta}_0}= O(1/\sqrt{n}),
  \end{align}
  where $T$ is a random variable with continuous CDF.
  
\end{lemma}

\begin{proof}
    It is obvious that equation \eqref{eq:thm_iid_c_step1} implies \ref{assmp:cont}. Let $G_T$ denote the CDF of $T$, then
    \begin{align*}
        &\sup_{\bm{\theta} \in A_n(c)} \sup_{t \in \mbR} \abs{G_n(t \mid \bm{\theta})-G_n(t \mid \bm{\theta}_0)} \\
        \le& \sup_{t \in \mbR} \abs{G_n(t \mid \bm{\theta}_0)-G_T(t)} + \sup_{\bm{\theta} \in A_n(c)} \sup_{t \in \mbR} \abs{G_n(t \mid \bm{\theta})-G_T(t)}.
    \end{align*}
    Therefore to show equation \eqref{eq:thm_iid_c_step1} implies \ref{assmp:uc}, it suffices to show the last term converges to zero for any $c>0$. Suppose not, then there exists $c \in \mbR$ and a sequence $\{\bm{\theta}_n^* \} \subset A_n(c)$ such that $\sup_{t \in \mbR} \abs{G_n(t \mid \bm{\theta}_n^*)-G_T(t)} \not\to 0$, i.e., $T_n$ does not converge in probability to $T$ under $P_{\bm{\theta}_n^*}$. However, $\norm{\bm{\theta}_n-\bm{\theta}_0}= O(1/\sqrt{n})$, hence a contradiction.

\end{proof}

The assumptions in Theorem 1 of \cite{robins2000asymptotic} satisfy conditions of Lemmas \ref{lem:l1}-\ref{lem:l2l3}, therefore the following result discussed in their Remark 2 is a special case of our Proposition \ref{prop:paper5thm3.2}: For asymptotically normal test statistics, the ppp is asymptotically well calibrated if the asymptotic mean of the statistic is constant in the unknown parameter. 

\subsection{The classical Kolmogorov-Smirnov statistic}

Now we present sufficient conditions for the CKS to satisfy Lemmas \ref{lem:l1} and \ref{lem:l2l3}, such that ppp(CKS) is asymptotically well calibrated. Let $Y_1, \dots, Y_n$ be random variables on $(\mbR, \mcB(\mbR))$, where $\mcB$ denotes the Borel $\sigma$-algebra. Consider using ppp(CKS) with estimator $\hat{\bm{\theta}}_n$ to test for $H: Y_i \sim$ i.i.d. $P_{\bm{\theta}}$, $\bm{\theta} \in \Theta \subset \mbR^k$, where $P_{\bm{\theta}}$ is a parametric model with density $p_{\bm{\theta}}$ under the Lebesgue measure. Let $\dot{P}_{\bm{\theta}}: \mbR \to \mbR^k$ denote the gradient vector of $P_{\bm{\theta}}$ with respect to $\bm{\theta}$, $s(\bm{\theta})=s(Y;\bm{\theta}): \interior{\Theta} \to \mbR^k$ denote the score vector, and we use $s_i(\bm{\theta})=s(Y_i;\bm{\theta})$ for the score contribution of $Y_i$. Unless otherwise specified, convergence in probability and in distribution are under $P_{\bm{\theta}_0}$. We assume:

\begin{enumerate}[label=A\arabic{*}., ref=A\arabic{*}]
  \item\label{assmp:supp_A} The support of $p_{\bm{\theta}}$, $\{y \in \mbR; p_{\bm{\theta}}(y)>0 \}$, is independent of $\bm{\theta}$.
  \item\label{assmp:diff_A} For all $y$, $p_{\bm{\theta}}(y)$ is continuously differentiable in $\bm{\theta}$, and the map $\bm{\theta} \mapsto P_{\bm{\theta}}$ from $\mbR^k$ to $\ell^\infty(\mbR)$ is Fréchet differentiable at $\bm{\theta}_0$, i.e., there exists a bounded linear map $A_{\bm{\theta}_0}: \mbR^k \mapsto \ell^\infty(\mbR)$, referred to as the Fréchet derivative, such that
  \begin{align}\label{eq:frechet}
    \sup_{y \in \mbR} \abs{P_{\bm{\theta}_0+\Bh}(y)-P_{\bm{\theta}_0}(y) - A_{\bm{\theta}_0}(\Bh)(y)} = o(\norm{\Bh}) \quad \text{as } \Bh \to \bm{0}.
  \end{align}
  \item\label{assmp:domination_A} There exists $\delta_0>0$ such that $E[\sup_{\bm{\theta} \in B(\delta_0,\bm{\theta}_0)} \norm{s(\bm{\theta})}] < \infty$.
  \item\label{assmp:Fisher_A} Elements of the Fisher information matrix $I(\bm{\theta})=E_{\bm{\theta}}[s(\bm{\theta})s(\bm{\theta})^\top]$ are well-defined, continuous in $\bm{\theta}$, and non-singular for all $\bm{\theta} \in \Theta$.
  \item\label{assmp:BvM} Let $\hat{\bm{\theta}}_n^M$ denote the maximum likelihood estimator (MLE), $\sqrt{n}(\hat{\bm{\theta}}_n^{M}-\bm{\theta}_0)=O_p(1)$, and $d_{TV}(\pi(\cdot \mid \BY_n), \hat{\pi}_n) \overset{p}{\to} 0$, where $\hat{\pi}_n$ is the probability measure for $N(\hat{\bm{\theta}}^M_n, I(\bm{\theta}_0)^{-1}/n)$.
  \item\label{assmp:estimator} $\hat{\bm{\theta}}_n$ admits the following asymptotic linear expansion:
    \begin{align*}
        \hat{\bm{\theta}}_n = \bm{\theta}_0 + \frac 1n I(\bm{\theta}_0)^{-1} \sum_{i=1}^n s_i(\bm{\theta}_0) + o_p(\bm{1}_k/\sqrt{n}).
    \end{align*}
\end{enumerate}

\begin{theorem}\label{thm:CKS}
  Under \ref{assmp:supp_A}-\ref{assmp:estimator}, the ppp(CKS) is asymptotically well calibrated.
\end{theorem}

\begin{remark}\label{rmk:diff}
	\ref{assmp:diff_A} specifies the required differentiability conditions. These are typically satisfied by continuous distributions from the exponential family. The key requirements are:
	\begin{enumerate}[label=\roman*]
		\item Fréchet differentiability, for asymptotic tightness of $\sqrt{n}(P_n-P_{\hat{\bm{\theta}}_n})$ (Theorem 19.23, \cite{van2000asymptotic}).
		\item Quadratic mean differentiability (QMD) of the null model at $\bm{\theta}_0$, for local asymptotic normality to invoke Le Cam's third lemma.
	\end{enumerate}
\end{remark}

It is not straightforward to check Fréchet differentiability. We present a sufficient condition below which is easier to verify.

\begin{proposition}\label{prop:frechet}
    Let $Q(u,\bm{\theta})$ denote the quantile function of $P_{\bm{\theta}}$, then $g(u,\bm{\theta}):=\dot{P}_{\bm{\theta}}(Q(u,\bm{\theta}))$ is a function defined on $(0,1) \times \Theta$. If $g$ can be extended to a continuous function on $[0,1] \times \Theta$, then the map $\bm{\theta} \mapsto P_{\bm{\theta}}$ is Fréchet differentiable.
\end{proposition}

\begin{remark}
    For example, for the Gamma model considered in the simulation study in Section \ref{sec:sim_gamma}, it can be shown that $\lim_{u \to 0^+}g(u,\bm{\theta})=\bm{0}$, and $\lim_{u \to 1^-}g(u,\bm{\theta})=\bm{0}$. Consequently, the condition in Proposition \ref{prop:frechet} is satisfied.
\end{remark}

\begin{proof}
    This proposition makes reference to the proofs of Lemma 1 and Lemma 2 in \cite{durbin1973weak}. To show Fréchet differentiability of $P_{\bm{\theta}}$ at $\bm{\theta}_0$, it suffices to show:
    \begin{enumerate}[label=(\roman*)]
        \item $\dot{P}_{\bm{\theta}_0}$ is a bounded linear operator, and
        \item $\sup_{y \in \mbR} \abs{P_{\bm{\theta}_0+\Bh}(y)-P_{\bm{\theta}_0}(y)-\dot{P}_{\bm{\theta}_0}(y)^\top \Bh} = o(\norm{\Bh})$.
    \end{enumerate}
    By the assumption, $g$ can be extended to a function $\bar{g}$, a function continuous in $(u,\bm{\theta})$ on the compact set $[0,1] \times \bar{B}(\delta_0,\bm{\theta}_0)$, where $\bar{B}$ denotes a closed ball. Hence there exists $M_1>0$ such that $\sup_{u \in (0,1),\bm{\theta} \in B(\delta_0,\bm{\theta}_0)}\norm{g(u,\bm{\theta})} \le \sup_{u \in [0,1],\bm{\theta} \in \bar{B}(\delta_0,\bm{\theta}_0)}\norm{\bar{g}(u,\bm{\theta})} \le M_1$. Therefore
    \begin{align*}
        \sup_{y \in \mbR, \bm{\theta}\in B(\delta_0,\bm{\theta}_0)}\norm{\dot{P}_{\bm{\theta}}(y)}&= \sup_{u \in (0,1),\bm{\theta} \in B(\delta_0,\bm{\theta}_0)}\norm{\dot{P}_{\bm{\theta}}(Q(u,\bm{\theta}))} \\
        &= \sup_{u \in (0,1),\bm{\theta} \in B(\delta_0,\bm{\theta}_0)}\norm{g(u,\bm{\theta})} \le M_1,
    \end{align*}
    which is sufficient for (i).

    For (ii), observe that the left-hand side equals to $\sup_{y \in \mbR} \abs{(\dot{P}_{\tilde{\bm{\theta}}}(y)-\dot{P}_{\bm{\theta}_0}(y))^\top \Bh}$ for some $\tilde{\bm{\theta}}=t\bm{\theta}_0 + (1-t)(\bm{\theta}_0+\Bh)$, $t \in [0,1]$. Let $y=Q(u,\bm{\theta}_0)=Q(\tilde{u},\tilde{\bm{\theta}})$, observe 
    \begin{align*}
        \sup_{u \in (0,1)}\abs{\tilde{u}(u)-u} &= \sup_{u \in (0,1)}\abs{P_{\tilde{\bm{\theta}}}(Q(u,\bm{\theta}_0))-P_{\bm{\theta}_0}(Q(u,\bm{\theta}_0))} \\
        &= \sup_{u \in (0,1)}\abs{\dot{P}_{\bm{\theta}^*}(Q(u,\bm{\theta}_0))^\top (\tilde{\bm{\theta}}-\bm{\theta}_0)} \le M_1\norm{\tilde{\bm{\theta}}-\bm{\theta}_0},
    \end{align*}
    where $\bm{\theta}^* = t^*\bm{\theta}_0 + (1-t^*)\tilde{\bm{\theta}}$, for some $t^* \in [0,1]$. Therefore it goes to zero uniformly in $u$ as $\tilde{\bm{\theta}} \to \bm{\theta}_0$. And because $\bar{g}$ is uniformly continuous in $(u,\bm{\theta})$ on the compact set $[0,1] \times \bar{B}(\delta_0,\bm{\theta}_0)$, so is $g$ on $(0,1) \times B(\delta_0,\bm{\theta}_0)$. Hence as $\Bh \to \bm{0}$ and $\tilde{\bm{\theta}} \to \bm{\theta}_0$, $g(\tilde{u},\tilde{\bm{\theta}}) \to g(u,\bm{\theta}_0)$ uniformly in $u$. Therefore
    \begin{align*}
        \sup_{y \in \mbR}\norm{\dot{P}_{\tilde{\bm{\theta}}}(y)-\dot{P}_{\bm{\theta}_0}(y)} &= \sup_{u \in (0,1)}\norm{\dot{P}_{\tilde{\bm{\theta}}}(Q(u,\bm{\theta}_0))-\dot{P}_{\bm{\theta}_0}(Q(u,\bm{\theta}_0))} \\
        &= \sup_{u \in (0,1)}\norm{\dot{P}_{\tilde{\bm{\theta}}}(Q(\tilde{u},\tilde{\bm{\theta}}))-\dot{P}_{\bm{\theta}_0}(Q(u,\bm{\theta}_0))}\\
        &= \sup_{u \in (0,1)}\norm{g(\tilde{u},\tilde{\bm{\theta}}) - g(u,\bm{\theta}_0)} \to 0
    \end{align*}
    as $\Bh \to \bm{0}$, which completes the proof.
\end{proof}

\begin{remark}\label{rmk:domination}
	\ref{assmp:domination_A} ensures $s(\bm{\theta})$ is locally uniformly integrably bounded at $\bm{\theta}_0$, so that we can apply Leibniz's rule to interchange integral and differentiation below, which is needed in the proof:
	\begin{align*}
	\int_{-\infty}^y s(\bm{\theta}_0)p_{\bm{\theta}_0}(u)du = \int_{-\infty}^y \frac{\partial P_{\bm{\theta}}(u)}{\partial \bm{\theta}} \Big|_{\bm{\theta}=\bm{\theta}_0} du = \frac{\partial}{\partial \bm{\theta}}\int_{-\infty}^y P_{\bm{\theta}}(u) du \Big|_{\bm{\theta}=\bm{\theta}_0} = \dot{P}_{\bm{\theta}_0}(y).
	\end{align*}
	A sufficient condition for \ref{assmp:domination_A} is that for all $i$, there exists a continuous function $g_i: \Theta \to \mbR$ and a measurable function $h_i: \mbR \to \mbR$ such that
  \begin{align}\label{eq:domination}
    \abs{s_i(\bm{\theta})} \le g_i(\bm{\theta}) + h_i(Y), \quad E_{\bm{\theta}_0}[h_i(Y)] < \infty.
  \end{align}
	This is because by continuity of $g_i$, there exists $\bm{\theta}_i^* = \argmax_{\bm{\theta} \in \bar{B}(\delta_0,\bm{\theta}_0)} g_i(\bm{\theta})$, hence
  \begin{align*}
    E_{\bm{\theta}_0}\sbr{\sup_{\bm{\theta} \in B(\delta_0,\bm{\theta}_0)} \abs{s_i(\bm{\theta})}} \le g_i(\bm{\theta}_i^*) + E_{\bm{\theta}_0}[h_i(Y)]< \infty,
  \end{align*}
  which is sufficient for \ref{assmp:domination_A}.
  
  \ref{assmp:Fisher_A} states standard regularity conditions on the Fisher information matrix. Overall \ref{assmp:domination_A} and \ref{assmp:Fisher_A} are mild. They hold for commonly used continuous distributions, including those in the exponential family. 
\end{remark}

\begin{remark}\label{rmk:BvM}
	\ref{assmp:BvM} is regularity condition on the MLE and the BvM property. By Lemma \ref{lem:l1}, this is sufficient for \ref{assmp:pc}.
\end{remark}

\begin{remark}\label{rmk:estimator}
	\ref{assmp:estimator} states a standard asymptotic expansion of the MLE; see, for example, Theorem 5.39 in \cite{van2000asymptotic}. This theorem, together with discussion in Remark \ref{rmk:domination}, suggests that a sufficient condition for \ref{assmp:estimator} is to have the functions $g_j$ and $h_j$ such that equation \eqref{eq:domination} holds, and $E[h_j(Y)^2] < \infty$. Consequently, any estimator that is asymptotically equivalent to the MLE can be used as $\hat{\bm{\theta}}_n$, including Bayesian estimators such as the posterior mean under \ref{assmp:BvM}.
\end{remark}

Following Remark \ref{rmk:BvM}, to prove Theorem \ref{thm:CKS}, it suffices to verify that the CKS statistic satisfies equation \eqref{eq:thm_iid_c_step1} in Lemma \ref{lem:l2l3}. We begin with a Taylor expansion:
\begin{align*}
     Z_n(y)=\sqrt{n}(P_n(y)- P_{\hat{\bm{\theta}}_n}(y)) = \sqrt{n}(P_n(y)- P_{\bm{\theta}_0}(y)+\dot{P}_{\bm{\theta}_0}(y)^\top (\bm{\theta}_0-\hat{\bm{\theta}}_n)) + r_n(y).
\end{align*}
Fréchet differentiability of $P_{\bm{\theta}}$ at $\bm{\theta}_0$ ensures that the remainder term $r_n(y)=o_p(1)$ uniformly in $y$, and hence it does not affect the asymptotic distribution of $Z_n$. Using Le Cam's third lemma, we show that $Z_n(y)$ satisfies equation \eqref{eq:thm_iid_c_step1} for each fixed $y \in \mbR$. We then establish asymptotic tightness of $Z_n$ in $l^\infty(\mbR)$, which ensures that equation \eqref{eq:thm_iid_c_step1} holds uniformly in $y$, and consequently for the CKS statistic $K_n = \sup_{y \in \mbR} \abs{Z_n(y)}$. See Appendix \ref{app:CKS} for the full proof.

\subsection{The generalized Kolmogorov-Smirnov statistic}

Next, we present sufficient conditions for the ppp(GKS) to be asymptotically well calibrated under the regression setting, where each observation $Y_i$ is associated with known covariate vector $\Bx_i \in \mcX \subset \mbR^p$. Let $P_{i,\bm{\theta}}$, $i=1,\dots,n$ be probability measures on $(\mbR_i,\mcB_i)$, with $(\mbR_i,\mcB_i)=(\mbR,\mcB)$. We assume they are all dominated by the Lebesgue measure and let $p_{i,\bm{\theta}}$ denote the density. Set $(\mcY,\mcA)=\prod_{i=1}^\infty (\mbR_i,\mcB_i)$ and let $P_{\bm{\theta}}$ be the product measure of $P_{i,\bm{\theta}}$'s induced on $\mcA$. Let $Y_i$'s be coordinate random variables, it then follows that they are independently distributed. Further define $\mcA_n =\sigma(Y_1,\dots,Y_n)$ as the $\sigma$-algebra induced by the random variables $Y_1,\dots,Y_n$, and let $P_{\bm{\theta}}^n$ denote the restriction of $P_{\bm{\theta}}$ to $\mcA_n$. Consider testing for: $H: Y_i \mid \Bx_i \sim P_{i,\bm{\theta}}$, $\bm{\theta} \in \Theta \subset \mbR^k$, with $P_{i,\bm{\theta}}(y)=F(y; \Bx_i, \bm{\theta})$, for some parametric CDF $F$, with densities $p_{i,\bm{\theta}}(y)=f(y;\Bx_i,\bm{\theta})$. Let $Q(u;\Bx_i,\bm{\theta})=F(\cdot;\Bx_i,\bm{\theta})^{-1}(u)$, $u \in (0,1)$, denote the corresponding quantile function. We define functions $h_{i,\bm{\theta},u}: \mbR_i \mapsto \{-1,0,1\}$,
\begin{align*}
  h_{i,\bm{\theta},u}(y_i) := 1(F(y_i;\Bx_i,\bm{\theta}) \le u) - 1(F(y_i;\Bx_i,\bm{\theta}_0) \le u),
\end{align*}
and coordinate-wise vector function $h_{\bm{\theta},u}= (h_{1,\bm{\theta},u}, \dots, h_{n,\bm{\theta},u}): \mbR^n \mapsto \{-1,0,1\}^n$. For any $0<\delta<\delta_0$, define class of vector functions: 
\begin{align}\label{eq:Hn}
  \mcH_n(\delta) &:= \{ h_{\bm{\theta},u}: \bm{\theta} \in B(\delta,\bm{\theta}_0), u \in [0,1] \}.
\end{align}
When the dimension $n$ is clear from context, we write $\mcH(\delta)$ for simplicity. Let $\hat{d}_n$ denote the empirical $L_2$ semi-metric,
\begin{align*}
  \hat{d}_n(h_{\bm{\theta}, u},h_{\bm{\theta}', u'}) = \cbr{\frac 1n \sum_{i=1}^n (h_{i,\bm{\theta},u}(Y_i)-h_{i,\bm{\theta}',u'}(Y_i))^2}^{1/2}.
\end{align*}
For any $\epsilon>0$, a collection of functions $\mcF$ and a semi-metric $d$ on $\mcF$, $N(\epsilon,\mcF,d)$ denotes the covering number of $\mcF$ under $d$. We use $E^*$ for expectation in $P_{\bm{\theta}_0}$ outer probability, and $\bar{E}$ for expectation in the completion of $P_{\bm{\theta}_0}$. We further assume:

\begin{enumerate}[label=A\arabic{*}., ref=A\arabic{*}, start=7]
  \item\label{assmp:supp_B} For all $\Bx \in \mcX$, the support of $f$, $\{y \in \mbR; f(y;\Bx,\bm{\theta})>0 \}$, is independent of $\bm{\theta}$.
  \item\label{assmp:diff_B} For all $y \in \mbR$ and $\Bx \in \mcX$, $f(y;\Bx,\bm{\theta})$ is continuously differentiable in $\bm{\theta}$. The vector-valued functions $g_i(u,\bm{\theta})$ defined below exist and can be extended to continuous functions in $(u,\bm{\theta})$ on $[0,1] \times \bar{B}(\delta_0,\bm{\theta}_0)$ uniformly in $i$:
  \begin{align*}
      g_i(u,\bm{\theta}) := \dot{F}(Q(u;\Bx_i,\bm{\theta});\Bx_i,\bm{\theta}), \quad \text{where } \dot{F}(y;\Bx_i,\bm{\theta})=\frac{\partial F(y;\Bx_i,\bm{\theta})}{\partial \bm{\theta}}.
  \end{align*}
  And there exists $g(u)$ such that $1/n \sum_{i=1}^n g_i(u,\bm{\theta_0}) \to g(u)$ uniformly in $u$.
  \item\label{assmp:domination_B} There exists $\delta_0>0$ such that $E[\sup_{\bm{\theta} \in B(\delta_0,\bm{\theta}_0)} \norm{s_i(\bm{\theta})}] < \infty$ for all $i$.
  \item\label{assmp:Fisher_B} For all $i$, elements of the Fisher information matrix $\BI_i(\bm{\theta})=E[s_i(\bm{\theta})s_i(\bm{\theta})^\top]$ are well defined, continuous in $\bm{\theta}$. For all $\bm{\theta} \in \Theta$, $\BI_i(\bm{\theta})$'s are non-singular, and there exists positive definite matrix $\BI(\bm{\theta}) \in \mbR^{k \times k}$ such that $1/n \sum_{i=1}^n \BI_i(\bm{\theta}) \to \BI(\bm{\theta})$.
  \item\label{assmp:ui} There exists $\gamma>0$ such that $\sup_{i \in \mbN} E\sbr{\norm{s_i(\bm{\theta}_0)}^{2+\gamma}} < \infty$.
  \item\label{assmp:entropy} There exist $K_1$, $K_2>0$ such that $N(\epsilon,\mcH_n(\delta),\hat{d}_n) \le K_1 n^{K_2}$ for all $\epsilon>0$, and $\delta \in (0,\delta_0)$.
\end{enumerate}

\begin{theorem}\label{thm:GKS}
  Under \ref{assmp:BvM}-\ref{assmp:entropy}, the ppp(GKS) is asymptotically well calibrated.
\end{theorem}

\begin{remark}
    Some of the assumptions of Theorem \ref{thm:GKS} parallel those of Theorem \ref{thm:CKS}. In particular, \ref{assmp:BvM} and \ref{assmp:estimator} remain the same, whereas \ref{assmp:supp_B}-\ref{assmp:Fisher_B} strengthen \ref{assmp:supp_A}-\ref{assmp:Fisher_A} by imposing additional regularity conditions on behavior across observations. We highlight that for independent but non-identically distributed random variables, Fréchet differentiability of the maps $\bm{\theta} \mapsto P_{i,\bm{\theta}}$ is insufficient for the proof. This motivates the conditions imposed in \ref{assmp:diff_B}. As shown in Proposition \ref{prop:frechet}, these are stronger assumptions than Fréchet differentiability. 

    For situations where \ref{assmp:supp_A}-\ref{assmp:Fisher_A} hold, these additional conditions are typically satisfied if it is fitting to assume $(\BX_i,Y_i)$'s are independently and identically distributed. 
\end{remark}

\begin{remark}
    \ref{assmp:ui} is needed to invoke Theorem 3.1 in \cite{philippou1973asymptotic} to obtain asymptotic normal expansion of the log-likelihood ratio for independent but not identically distributed observations.
\end{remark}

In Propositions \ref{prop:location_scale} and \ref{prop:o-minimal}, we present several sufficient conditions for \ref{assmp:entropy}. 

\begin{proposition}\label{prop:location_scale}
    \ref{assmp:entropy} holds if $F$ is a location-scale family and
    \begin{enumerate}[label=(\roman*)]
        \item the regression term $\Bx^\top \bm{\beta}$ is only in the location parameter: $F(y;\Bx,\bm{\beta},\sigma,\bm{\nu}) = G_{\bm{\nu}} [(y-h(\Bx^\top \bm{\beta}))/\sigma]$, or
        \item the regression term $\Bx^\top \bm{\beta}$ is only in the scale parameter: $F(y;\Bx,\bm{\beta},\mu,\bm{\nu}) = G_{\bm{\nu}} [(y-\mu)/h(\Bx^\top \bm{\beta})]$,
    \end{enumerate}
    for some CDF $G$ that is absolutely continuous with respect to Lebesgue measure, parameterized by $\bm{\nu} \in \mbR^d$, and strictly monotone function $h$.
\end{proposition}

\begin{remark}
    Proposition \ref{prop:location_scale} covers generalized linear models (GLMs) such as normal linear regression, Student-t regression, Gamma GLM and Weibull GLM. Lognormal regression can also be covered with a slight modification to the proof in Appendix \ref{app:lem_prop}. However, it does not cover those that are not location-scale families such as Beta regression and inverse Gaussian GLM -- they are covered by Proposition \ref{prop:o-minimal}.
\end{remark}

\begin{proposition}\label{prop:o-minimal}
    \ref{assmp:entropy} holds if for any compact $A \subset \mbR$, such that $[0,1] \subset A$ and $B(\delta,\bm{\theta}_0) \subset A^k$, the collection of sets $\mcS_A$ defined below is a subset of a uniformly definable family in an o-minimal structure.
    \begin{align*}
      \mcS_A = \{ \{(\Bx,y) \in A^p \times A: F(\Bx,y,\bm{\theta}) \le u\}, \bm{\theta} \in B(\delta,\bm{\theta}_0), u \in [0,1] \}.
    \end{align*}
\end{proposition}

\begin{remark}
    Proposition \ref{prop:o-minimal} relies on o-minimality theory and sample compression schemes. We defer definitions, references and the proof to Appendix \ref{app:lem_prop}. As formally stated in Corollary \ref{coro:regression} below, this condition holds for all the common regression models under continuous distributions. See Appendix \ref{app:coro} for the proof.
\end{remark}

\begin{corollary}\label{coro:regression}
    Under assumptions \ref{assmp:BvM}-\ref{assmp:ui}, common regression models with continuous distributions satisfy conditions in Propositions \ref{prop:location_scale} or \ref{prop:o-minimal} such that ppp(GKS) is asymptotically well calibrated. This includes but not limited to: normal linear regression, Student-t regression, Lognormal regression, Gamma GLM, Weibull GLM, Beta GLM, and inverse Gaussian GLM.
\end{corollary}

The proof of Theorem \ref{thm:GKS} follows the same strategy as that of Theorem \ref{thm:CKS}, but is technically more involved. The difficulty arises from the process
\begin{align*}
    \tilde{Z}_n(u) = \sqrt{n} \cbr{u - \frac 1n \sum_{i=1}^n 1(F(Y_i;\Bx_i,\hat{\bm{\theta}}_n) \le u)},
\end{align*}
which depends on $\hat{\bm{\theta}}_n$ through discontinuous indicator functions. Consequently, Taylor expansion is not applicable. Instead, we decompose $\tilde{Z}_n(u)$ as
\begin{align}\label{eq:GKS_decompose}
  &\tilde{Z}_n(u) = \underbrace{\frac 1{\sqrt{n}}\sum_{i=1}^n \cbr{u - 1(F(Y_i;\Bx_i,\bm{\theta}_0) \le u)}}_{\text{(a)}} + \underbrace{\frac 1{\sqrt{n}} \sum_{i=1}^n (m_i(\bm{\theta}_0, u) - m_i(\hat{\bm{\theta}}_n, u))}_{\text{(b)}} + \nonumber \\
  &\underbrace{\frac 1{\sqrt{n}} \sum_{i=1}^n (1(F(Y_i;\Bx_i,\bm{\theta}_0) \le u)- m_i(\bm{\theta}_0, u) - 1(F(Y_i;\Bx_i,\hat{\bm{\theta}}_n) \le u) + m_i(\hat{\bm{\theta}}_n, u))}_{\text{(c)}},
\end{align}
where $m_i(\bm{\theta},u):=F(Q(u;\Bx_i,\bm{\theta});\Bx_i,\bm{\theta}_0)$. The first step is to show that term $(c)=o_p(1)$ uniformly in $u$. To this end, let $\mbG_n$ denote the empirical process,
\begin{align*}
  \mbG_n(h_{\bm{\theta},u}) := \frac 1{\sqrt{n}} \sum_{i=1}^n [h_{i,\bm{\theta},u}(Y_i) - E h_{i,\bm{\theta},u}(Y_i)].
\end{align*}
Let $\{ \delta_n \} \subset (0,\delta_0)$ be a decreasing sequence satisfying $\delta_n \downarrow 0$, and $\sqrt{n}\delta_n \to \infty$: For example, we will use $\delta_n = O(\log(n)/\sqrt{n})$ in the proofs. Then under regularity assumptions, term (c) is $o_p(1)$ uniformly in $u$ if $E^*[\norm{\mbG_n(h_{\bm{\theta},u})}_{\mcH(\delta_n)}] \to 0$. We establish this using standard empirical process tools such as symmetrization and Dudley's entropy integral theorem, and we bound the entropy numbers using \ref{assmp:entropy}. Then we show term (a)+(b) satisfies equation \eqref{eq:thm_iid_c_step1} of Lemma \ref{lem:l2l3} for each fixed $u$, and is asymptotically tight such that equation \eqref{eq:thm_iid_c_step1} holds for the GKS statistic $\tilde{K}_n = \sup_{u \in [0,1]}\abs{\tilde{Z}_n(u)}$. See Appendix \ref{app:GKS} for the full proof.

\section{Numerical experiments}\label{sec:sim}

We conducted simulation studies to evaluate the finite-sample performance of the ppp(CKS) and ppp(GKS). In Section \ref{sec:sim_gamma}, ppp(CKS) is used to assess a Gamma model; In Section \ref{sec:sim_gglm}, ppp(GKS) is applied to a Gamma GLM. 

We first study the null distributions. To examine the sensitivity of finite-sample performance to prior specification, we considered two sets of priors across a range of sample sizes. Specifically, we included reasonable priors such as weakly informative priors or priors loosely centered at the true parameter values, as well as misspecified informative priors that place the true parameters in the tails of the distributions. We further compared Bayesian and frequentist plug-in estimators to evaluate their impact on the resulting ppp. We then study the power of the ppp(CKS) and ppp(GKS). In each case, we simulated data from three alternative models --  Although specific alternative models were considered here, the goal was not model selection. Rather, we wanted to assess whether KS-type statistics are sensitive to different types of model misspecification, and how their performance compare with other test statistics such as the chi-squared test.

Our simulation results showed that both the ppp(CKS) and ppp(GKS) are well behaved under the null across different priors and estimators, and exhibit good power against a variety of alternative models.

\subsection{Gamma model}\label{sec:sim_gamma}

Consider using ppp(CKS) to test for $H$: $Y_i \sim$ i.i.d. $\mathrm{Gamma}(\alpha,\beta)$ (shape, rate parameterization). In order to study its null distribution, we simulated data from $\mathrm{Gamma}(\alpha_0=2,\beta_0=5)$, and we considered two estimators: the MLE, and the posterior mean, four sample sizes: $n=10,20,100,500$, and two sets of priors:
\begin{enumerate}[label=(\roman*)]
  \item Good priors: loosely centered at the true parameters
  \begin{itemize}
    \item $\pi(\alpha) \sim TN(2.5,16,0,\infty)$, a normal with mean 2.5, variance 16, truncated to the positive part of the real line;
    \item $\pi(\beta) \sim \mathrm{Gamma}(1,1)$.
  \end{itemize}
  \item Bad priors: chosen such that $\alpha_0$ and $\beta_0$ are roughly at the 97.5\% and 2.5\% percentiles respectively
  \begin{itemize}
    \item $\pi(\alpha) \sim TN(1,0.5,0,\infty)$;
    \item $\pi(\beta) \sim \mathrm{Gamma}(3,1.25)$.
  \end{itemize}
\end{enumerate}
\ref{assmp:supp_A}-\ref{assmp:estimator} hold under this setup, hence Theorem \ref{thm:CKS} applies, and ppp(CKS) is asymptotically well calibrated. To see why the assumptions are satisfied: \ref{assmp:supp_A} and the continuous differentiability of the density function in \ref{assmp:diff_A} obviously hold. It can then be shown that the condition in Proposition \ref{prop:frechet} holds which is sufficient for Fréchet differentiability. Gamma distribution belongs to the exponential family, and all the standard regularity conditions such as \ref{assmp:domination_A} and \ref{assmp:Fisher_A} hold; The MLEs are CAN estimators and admit the expansion in \ref{assmp:estimator}; All the priors we considered are continuous distributions with positive density over the entire parameter space, hence BvM theorem applies. Consequently, \ref{assmp:BvM} holds and the posterior means also satisfy the expansion in \ref{assmp:estimator}. See Appendix \ref{app:sim} for a detailed proof.

To carry out the numerical experiments, we simulated 1,000 datasets using $\alpha_0$ and $\beta_0$ for each sample size. Then for each simulated dataset and prior, we performed Bayesian model fitting in \texttt{rstan} to obtain 1,000 posterior samples (1,000 burn-in, followed by 5,000 iterations, thinned every 5 iterations). We then generated posterior predictive datasets, and computed the CKS test statistic for both the observed and posterior predictive datasets to approximate the ppp(CKS). Among these steps, posterior sampling is generally the most computationally demanding but is required for Bayesian inference regardless. The remaining steps are relatively straightforward for most parametric models. This illustrates the computational efficiency of the ppp and its natural compatibility with the Bayesian workflow.

Figure \ref{fig:gamma_level} presents the kernel density estimates of the sampling distribution of the ppp(CKS) under different sample sizes, priors and plug-in estimators. Under the good prior, the null distributions of the ppp(CKS) closely resemble uniform, even with a sample size as small as $n=10$. When the priors are misspecified, the posterior means are biased towards the prior, particularly for small sample sizes. This bias increases the discrepancy between the empirical distribution of the data and the fitted distribution, resulting in larger values of the CKS statistic. The discrepancy is more pronounced for the observed data than for the posterior predictive datasets, which leads to smaller $p$-values, as shown in the second plot. However, as $n$ increases to 100, these effects are largely washed away by the data. With $n=500$, the null distribution of the ppp(CKS) under the bad prior is again approximately uniform.

\begin{figure}
  \centering
  \includegraphics[width=15cm]{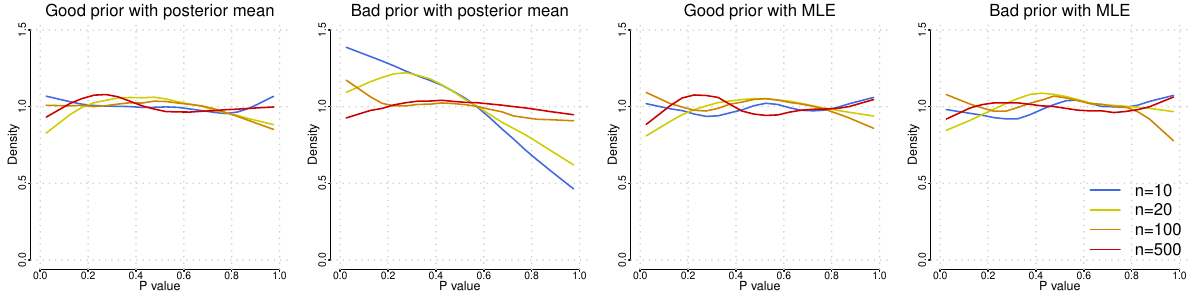}
  \caption{Gamma model example. Kernel density estimates of the null distributions of the ppp(CKS) under two priors, two estimators, and four sample sizes.}
  \label{fig:gamma_level}
\end{figure}

Using frequentist estimators such as the MLE in the CKS test may offer advantages over Bayesian estimators like the posterior mean. If the MLE can be easily estimated, it eliminates the need for Bayesian model fitting on the posterior predictive datasets, hence is more computationally efficient. Additionally, as shown in the last plot of Figure \ref{fig:gamma_level}, MLE is not affected by poor priors, which makes the results more robust to prior misspecification. For complex problems where the MLE is difficult to compute reliably, Bayesian estimators remain a viable alternative.

We further studied the power of the ppp(CKS) under three alternative models:
\begin{enumerate}[label=(\roman*)]
  \item Weibull model with data generated using shape parameter $\alpha_0$, and scale parameter $1/\beta_0$.
  \item Lognormal model with data generated using $\mu=0$, $\sigma=0.5$.
  \item Gamma GLM, 
  \begin{align}\label{eq:gamma_glm}
    Y_i \mid x_i \sim \mathrm{Gamma} \cbr{\alpha, \frac{\alpha}{x_i \theta + \alpha/\beta}},
  \end{align}
  with data generated using $\alpha = \alpha_0$, $\beta=\beta_0$, $\theta_0=0.5$, and $x_i$'s are known scalar covariates generated from Lognormal with parameters $\mu=0.5$, $\sigma=1$, then fixed for all datasets. Here the mean of $Y_i$ is $x_i \theta+\alpha/\beta$ instead of $\alpha/\beta$ under the Gamma model.
\end{enumerate}
We considered two additional test statistics for comparison with the CKS statistic. Let ($\hat{\alpha}_n$,$\hat{\beta}_n$) denote plug-in estimators for ($\alpha$,$\beta$):
\begin{enumerate}[label=(\roman*)]
  \item Chi-squared test, another popular omnibus test. Under $H$, it is defined as: $\sum_{i=1}^n (Y_i - \hat{\alpha}_n/\hat{\beta}_n)^2/[n\hat{\alpha}_n/\hat{\beta}_n^2]$. Deviation from the null is indicated by extreme values in either direction. Consequently, the corresponding $p$-value is two-sided, with both small and large values providing evidence against the null.
  \item Score test. The GLM in equation \eqref{eq:gamma_glm} is a larger parametric model $P_{\alpha,\beta,\theta}$ which contains the null model $P_{\alpha,\beta}$, and we recover $H$ when $\theta=0$. In this situation, \cite{robins2000asymptotic} and \cite{wang2021calibration} showed that the following score test exhibits particularly good power against the larger model:
  \begin{align*}
    \frac 1n \frac{\partial \ell_n(\hat{\alpha}_n,\hat{\beta}_n,\theta)}{\partial \theta} \Bigg|_{\theta=0} = \frac{\hat{\beta}_n^2}{\hat{\alpha}_n}\sum_{i=1}^n x_iY_i - \hat{\beta}_n \sum_{i=1}^n x_i,
  \end{align*}
  where $\ell_n$ denote the log-likelihood based on $n$ observations. Larger values suggest greater discrepancy between the null and the data.
\end{enumerate}
We used the good priors and $n=100$ for simulations. Figure \ref{fig:gamma_power} presents kernel density estimates of the sampling distribution of the ppp under these three test statistics across different models with the posterior mean as the plug-in estimator. Results with the MLE are very similar, hence omitted here.

\begin{figure}
  \centering
  \includegraphics[width=15cm]{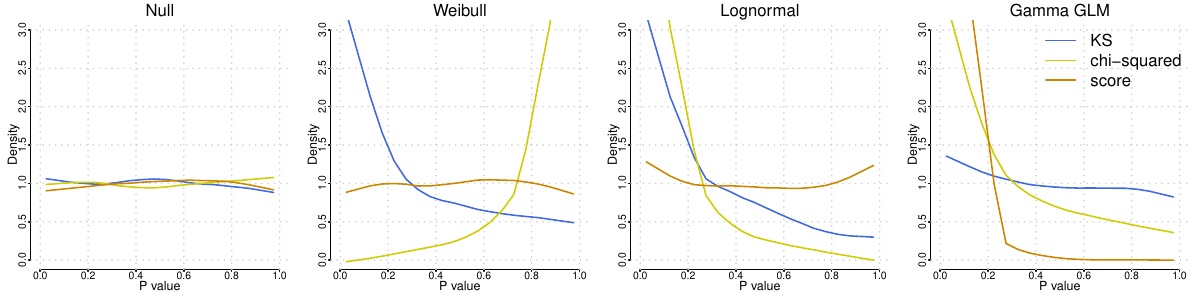}
  \caption{Gamma model example. Kernel density estimates of distributions of the ppp under different test statistics and data-generating models.}
  \label{fig:gamma_power}
\end{figure}
The first plot displays the null distributions, all of which are approximately uniform, indicating that the ppp is well-behaved under all three statistics. Under alternative models, we seek high power, which translates to right-skewed distributions under the CKS test or the score test, and either right-skewed or left-skewed distributions under the chi-squared test. The last plot shows results under the Gamma GLM. While all three test statistics are able to detect model misspecification, the score test is the most powerful as it is designed for this alternative. However, it lacks sensitivity to other types of alternatives as illustrated in the other two plots. This shows that, although omnibus statistics such as the CKS and the chi-squared statistics may not be most powerful against a specific alternative, they are general-purpose tools that can effectively detect a broad range of model misspecification. 

\subsection{Gamma GLM}\label{sec:sim_gglm}

Consider using the ppp(GKS) to test for $H$: $Y_i \mid \Bx_i \sim \mathrm{Gamma}(\alpha,b_i)$, $\log(\alpha/b_i)=\Bx_i^\top \bm{\beta}$, $\alpha \in \mbR$, $\bm{\beta} \in \mbR^k$. We used simulations to study the null distribution of the ppp(GKS), with $\alpha_0=2$, $p=6$. The covariates $x_{ij}$'s were generated from $N(0,1)$ except for the intercept column, and the regression coefficients $\beta_{0j}$'s were generated from $\mathrm{Unif}(-1,1)$. These values were then fixed for all the simulated datasets. We considered two estimators: the MLE, and the posterior mean, four sample sizes: $n=20,50,100,500$, and two sets of priors:
\begin{enumerate}[label=(\roman*)]
  \item Good priors: 
  \begin{itemize}
      \item $\pi(\bm{\theta}) \sim N(\mu,\sigma^2\BI)$, $\pi(\mu) \sim N(0,4)$: loosely centered at the true parameters.
      \item Half-Cauchy for both $\sigma^2$ and $\alpha$: weakly informative priors.
  \end{itemize}
  \item Bad priors:
  \begin{itemize}
      \item $\pi(\mu) \sim N(1.34,0.67)$: informative prior with the value $0$ at about 5 percentile;
      \item $\pi(\sigma^2) \sim \mathrm{IG}(3,14)$, $\pi(\alpha) \sim \mathrm{Gamma}(12,3)$: both set up so that the true parameters are at about 5 percentiles of the prior distributions.
  \end{itemize}
\end{enumerate}

Under regularity assumptions on $\Bx_i$'s, \ref{assmp:BvM}-\ref{assmp:entropy} hold, therefore Theorem \ref{thm:GKS} applies, and the ppp(GKS) is asymptotically well calibrated. Under our setup, it is plausible to view $(\BX_i,Y_i)$ as randomly drawn from the same distribution, therefore \ref{assmp:BvM}-\ref{assmp:Fisher_B} hold with similar arguments as in the last example. It can then be shown that \ref{assmp:ui} holds for the Gamma GLM, and \ref{assmp:entropy} holds by Corollary \ref{coro:regression}. See Appendix \ref{app:sim_gglm} for regularity assumptions on $\Bx_i$'s and the proof.

Following the same setup as in the Gamma model example, we conducted numerical experiments using 1,000 simulated datasets. Figure \ref{fig:gglm_level} shows the kernel density estimates of the null distributions of the ppp(GKS) under different sample sizes, priors and plug-in estimators.

Observations are similar to those from the Gamma model. Under the good prior, the null distributions of the ppp(GKS) closely resemble uniform, even with a sample size as small as $n=20$. When the priors are misspecified, the null distributions are distorted. But as $n$ increases, these effects are largely washed away by the data. Frequentist estimators are not affected by bad priors, hence results using the MLE are more robust to prior misspecification. However, as we shall see below, we might not always be able to reliably compute the frequentist estimators.

\begin{figure}
  \centering
  \includegraphics[width=15cm]{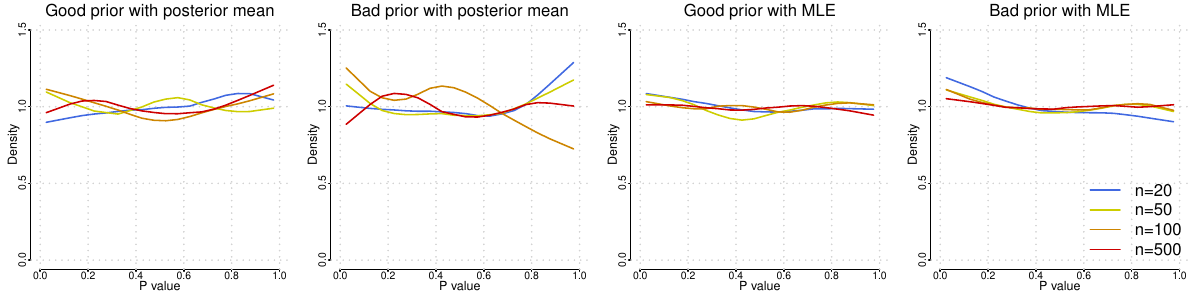}
  \caption{Gamma GLM example. Kernel density estimates of the null distributions of the ppp(GKS) under two priors, two estimators, and four sample sizes.}
  \label{fig:gglm_level}
\end{figure}

We further studied the power of the ppp(GKS) under three alternative models:
\begin{enumerate}[label=(\roman*)]
  \item Wrong link where a truncated log link $\log(m_i) = \min(0,\Bx_i^\top\bm{\beta})$ was used to generate the data.
  \item Lognormal model with data generated using $\mu=\Bx_i^\top \bm{\beta}$, $\sigma=\log(1/\alpha+1)$. These parameters were setup so that the first two moments are roughly the same as under the null.
  \item Missing covariate where the true mean is $\log(\alpha/b_i)=\Bx_i^\top \bm{\beta} + z_i \theta$, with $\theta=0.1$, for some known covariates $z_i$'s. 
\end{enumerate}
Similar to the Gamma example, we considered two other test statistics for comparison with the GKS statistic. Let ($\hat{\alpha}_n$, $\hat{\bm{\beta}}_n$) denote plug-in estimators for ($\alpha$, $\bm{\beta}$):
\begin{enumerate}[label=(\roman*)]
  \item Chi-squared test: $\sum_{i=1}^n [Y_i/\exp(\Bx_i^\top \hat{\bm{\beta}}_n)-1]^2\hat{\alpha}_n/n$. The corresponding $p$-value is two-sided.
  \item Score test. Taking alternative model (iii) as a larger parametric model, we can derive a test statistic with particularly good power against the embedding model: $\sum_{i=1}^n \hat{\alpha}_n z_i [Y_i/\exp(\Bx_i^\top \hat{\bm{\beta}}_n)-1]$.
\end{enumerate}
We used the posterior mean as the plug-in estimator, the good priors and $n=100$ for simulations. We did not use the MLE because estimation can be unstable and may even fail to converge for posterior predictive datasets under alternative models. This highlights the robustness of Bayesian estimators to plug in the KS-type statistics.

Figure \ref{fig:gglm_power} reports kernel density estimates of the sampling distribution of the ppp under these test statistics across different models. From the first plot, we can see that the ppp is well behaved under all three test statistics under the null. The last plot shows results under the missing covariate model. Because we set $\theta=0.1$, the signal is rather weak, and consequently both the GKS and the chi-squared statistics are not able to detect the wrong model, while the score test designed for this alternative model is the most powerful. However, the score test is not sensitive to other types of model misspecification, while the GKS and the chi-squared statistics are able to detect the Lognormal model, and the model with wrong link. 

\begin{figure}
  \centering
  \includegraphics[width=15cm]{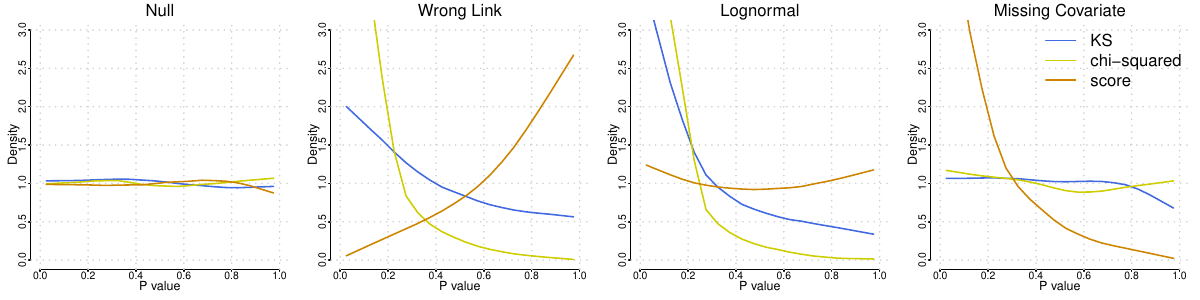}
  \caption{Gamma GLM example. Kernel density estimates of distributions of the ppp under different test statistics and data-generating models.}
  \label{fig:gglm_power}
\end{figure}

\section{Discussion}\label{sec:dis}

The posterior predictive $p$-value is a convenient tool for Bayesian model evaluation, but its conservatism has raised concerns among practitioners. \cite{robins2000asymptotic} showed that the ppp can be asymptotically well calibrated for a class of asymptotically normal test statistics, provided the asymptotic mean of the test statistic is constant in the parameter $\bm{\theta}$. Our work extends this theory to Kolmogorov-Smirnov (KS)-type test statistics: The classical KS (CKS) test is applicable to i.i.d. real-valued random variables, whereas the generalized KS (GKS) test can be used for regression setup. They are omnibus goodness-of-fit tests that are not asymptotically normal. We leverage their connection to asymptotic normality to establish the asymptotic well-calibration of the ppp through Le Cam's third lemma and empirical process theory. Moreover, our Proposition \ref{prop:paper5thm3.2} provides a pathway for potentially extending these results to other test statistics that are not asymptotically normal. We further used numerical experiments to show that both the ppp(CKS) and ppp(GKS) are well behaved under finite sample, and that they are versatile to detect a variety of alternative models.

Over the years, several alternative Bayesian $p$-value procedures have been developed to address the conservatism issue of the ppp. \cite{bayarri2000p} introduced the conditional predictive $p$-value and the partial posterior predictive $p$-value which are asymptotically well calibrated, as demonstrated by \cite{robins2000asymptotic}; \cite{hjort2006post} proposed calibrating the ppp using the prior predictive distribution, while \cite{wang2021calibration} suggested calibration with the posterior predictive distribution to retain asymptotic uniformity; More recently, \cite{moran2019population} and \cite{li2022calibrated} independently proposed using separate subsets of the data to compute the posterior and measure model discrepancy, addressing the double-dipping issue of the ppp. Although these alternative procedures often improve calibration and reduce conservatism, which typically leads to higher power, they generally do so at the cost of increased implementation complexity, higher computational burden, or reduced efficiency in information usage due to data splitting. With this article, we highlight the important role that test statistics play in the properties of $p$-values. For statistics such as the CKS and GKS tests, and those discussed in \cite{robins2000asymptotic}, practitioners can confidently use the ppp without resorting to more sophisticated procedures. It remains an active area of research to develop new procedures and characterize test statistics that yield asymptotically well-calibrated Bayesian $p$-values with good power.

\section*{Supplement}
Codes for replicating the simulation results are available at \url{https://github.com/christineymshen/BME_KS}

\bibliographystyle{plainnat}
\bibliography{ref}

\begin{thebibliography}{20}
\providecommand{\natexlab}[1]{#1}
\providecommand{\url}[1]{\texttt{#1}}
\expandafter\ifx\csname urlstyle\endcsname\relax
  \providecommand{\doi}[1]{doi: #1}\else
  \providecommand{\doi}{doi: \begingroup \urlstyle{rm}\Url}\fi

\bibitem[Bayarri and Berger(2000)]{bayarri2000p}
MJ~Bayarri and James~O Berger.
\newblock P values for composite null models.
\newblock \emph{Journal of the American Statistical Association}, 95\penalty0 (452):\penalty0 1127--1142, 2000.

\bibitem[Dudley(2006)]{dudley2006course}
Richard~M Dudley.
\newblock A course on empirical processes.
\newblock In \emph{Ecole d'{\'e}t{\'e} de Probabilit{\'e}s de Saint-Flour XII-1982}, pages 1--142. Springer, 2006.

\bibitem[Durbin(1973)]{durbin1973weak}
James Durbin.
\newblock Weak convergence of the sample distribution function when parameters are estimated.
\newblock \emph{The Annals of Statistics}, pages 279--290, 1973.

\bibitem[Fahrmeir and Kaufmann(1985)]{fahrmeir1985consistency}
Ludwig Fahrmeir and Heinz Kaufmann.
\newblock Consistency and asymptotic normality of the maximum likelihood estimator in generalized linear models.
\newblock \emph{The Annals of Statistics}, 13\penalty0 (1):\penalty0 342--368, 1985.

\bibitem[Floyd and Warmuth(1995)]{floyd1995sample}
Sally Floyd and Manfred Warmuth.
\newblock Sample compression, learnability, and the vapnik-chervonenkis dimension.
\newblock \emph{Machine learning}, 21\penalty0 (3):\penalty0 269--304, 1995.

\bibitem[Gelman et~al.(1996)Gelman, Meng, and Stern]{gelman1996posterior}
Andrew Gelman, Xiao-Li Meng, and Hal Stern.
\newblock Posterior predictive assessment of model fitness via realized discrepancies.
\newblock \emph{Statistica sinica}, pages 733--760, 1996.

\bibitem[Guttman(1967)]{guttman1967goodnessoffit}
Irwin Guttman.
\newblock The use of the concept of a future observation in goodness-of-fit problems.
\newblock \emph{Journal of the Royal Statistical Society. Series B (Methodological)}, 29\penalty0 (1):\penalty0 83--100, 1967.
\newblock ISSN 00359246.
\newblock URL \url{http://www.jstor.org/stable/2984569}.

\bibitem[He and Shao(1996)]{he1996general}
Xuming He and Qi-Man Shao.
\newblock A general bahadur representation of m-estimators and its application to linear regression with nonstochastic designs.
\newblock \emph{The Annals of Statistics}, 24\penalty0 (6):\penalty0 2608--2630, 1996.

\bibitem[Hjort et~al.(2006)Hjort, Dahl, and Steinbakk]{hjort2006post}
Nils~Lid Hjort, Fredrik~A Dahl, and Gunnhildur~H{\"o}gnad{\'o}ttir Steinbakk.
\newblock Post-processing posterior predictive p values.
\newblock \emph{Journal of the American Statistical Association}, 101\penalty0 (475):\penalty0 1157--1174, 2006.

\bibitem[Johnson and Laskowski(2010)]{johnson2010compression}
Hunter~R Johnson and Michael~C Laskowski.
\newblock Compression schemes, stable definable families, and o-minimal structures.
\newblock \emph{Discrete \& Computational Geometry}, 43\penalty0 (4):\penalty0 914--926, 2010.

\bibitem[Li and Huggins(2022)]{li2022calibrated}
Jiawei Li and Jonathan~H Huggins.
\newblock Calibrated model criticism using split predictive checks.
\newblock \emph{arXiv preprint arXiv:2203.15897}, 2022.

\bibitem[Moran et~al.(2019)Moran, Blei, and Ranganath]{moran2019population}
Gemma~E Moran, David~M Blei, and Rajesh Ranganath.
\newblock Population predictive checks.
\newblock \emph{arXiv preprint arXiv:1908.00882}, 2019.

\bibitem[Philippou and Roussas(1973)]{philippou1973asymptotic}
Andreas~N Philippou and GG~Roussas.
\newblock Asymptotic distribution of the likelihood function in the independent not identically distributed case.
\newblock \emph{The Annals of Statistics}, pages 454--471, 1973.

\bibitem[Robins et~al.(2000)Robins, van~der Vaart, and Ventura]{robins2000asymptotic}
James~M Robins, Aad van~der Vaart, and Val{\'e}rie Ventura.
\newblock Asymptotic distribution of p values in composite null models.
\newblock \emph{Journal of the American Statistical Association}, 95\penalty0 (452):\penalty0 1143--1156, 2000.

\bibitem[Rubin(1984)]{rubin1984bayesianly}
Donald~B Rubin.
\newblock Bayesianly justifiable and relevant frequency calculations for the applied statistician.
\newblock \emph{The Annals of Statistics}, pages 1151--1172, 1984.

\bibitem[Speissegger(1997)]{speissegger1997pfaffian}
Patrick Speissegger.
\newblock The pfaffian closure on an o-minimal structure.
\newblock \emph{arXiv preprint math/9710220}, 1997.

\bibitem[Van Den~Dries and Speissegger(2000)]{van2000field}
Lou Van Den~Dries and Patrick Speissegger.
\newblock The field of reals with multisummable series and the exponential function.
\newblock \emph{Proceedings of the London Mathematical Society}, 81\penalty0 (3):\penalty0 513--565, 2000.

\bibitem[Van~der Vaart(2000)]{van2000asymptotic}
Aad~W Van~der Vaart.
\newblock \emph{Asymptotic statistics}, volume~3.
\newblock Cambridge university press, 2000.

\bibitem[van~der Vaart and Wellner(2023)]{van2023weak}
A.W. van~der Vaart and J.A. Wellner.
\newblock \emph{Weak Convergence and Empirical Processes: With Applications to Statistics}.
\newblock Springer Series in Statistics. Springer, 2023.
\newblock ISBN 9783031290381.
\newblock URL \url{https://books.google.com/books?id=ZiwN0AEACAAJ}.

\bibitem[Wang and Xu(2021)]{wang2021calibration}
Zhendong Wang and Xingzhong Xu.
\newblock Calibration of posterior predictive p-values for model checking.
\newblock \emph{Journal of Statistical Computation and Simulation}, 91\penalty0 (6):\penalty0 1212--1242, 2021.

\end{thebibliography}

\newpage
\begin{appendix}

\section{Proofs of Lemmas and Propositions}\label{app:lem_prop}

We first state several supporting lemmas. Lemmas \ref{lem:cdf_conv}-\ref{lem:hole} are standard results, hence the proofs are omitted.

\begin{lemma}\label{lem:cdf_conv}
    Let $\{G_n\}$ be a sequence of CDFs converging pointwise to a continuous CDF $G$. The convergence is also uniform, i.e., $\norm{G_n-G}_\infty \to 0$.
\end{lemma}

\begin{lemma}\label{lem:uc}
    Let $\{g_n\}$ be a sequence of real functions converging uniformly to a real function $g$, $\{t_n\}$ be a sequence of real numbers converging to $t$. If $g$ is continuous, then $g_n(t_n) \to g(t)$.
\end{lemma}

\begin{lemma}\label{lem:hole}
    Let $X_n$ and $Y_n$ be sequences of real random variables. Let $G_n$ and $F_n$ be their CDFs. If $X_n$ converge in distribution to a random variable $X$ with continuous CDF $G$, then $E \abs{X_n-Y_n} \to 0 \implies \norm{G_n-F_n}_\infty \to 0$.
  
\end{lemma}

\begin{lemma}\label{lem:ias}
    For any $0<\delta<\delta_0$, $\mcH_n(\delta)$ defined in equation \eqref{eq:Hn} is image admissible Suslin (IAS), therefore $P_{\bm{\theta}}^n$-measurable. The same properties hold for $\mcH^2_n(\delta)$ defined below.
\begin{align}\label{eq:Hn2}
  \mcH_n^2(\delta) &:= \{ h_{\bm{\theta},u}^2 = (h_{1,\bm{\theta},u}^2, \dots, h_{n,\bm{\theta},u}^2); \bm{\theta} \in B(\delta,\bm{\theta}_0), u \in [0,1] \}.
\end{align}
    
\end{lemma}

\begin{proof}
    First we define Suslin space, Suslin sets and the IAS property following \cite{dudley2006course}. See the same paper for more details and references. 

    A separable measurable space $(Y,\mcS)$ is called a Suslin space if and only if there is a Polish space $X$ with Borel $\sigma$-algebra $\mcB$ and a measurable map from $X$ onto $Y$. A subset of a measurable space $Z$ is called a Suslin set if and only if it is a Suslin space with the relative $\sigma$-algebra. Given a measurable space $(X,\mcB)$, a collection $\mcF$ of measurable real functions on $X$ is called IAS via $(Y,\mcS,T)$ if and only if:
    \begin{enumerate}[label=(\roman*)]
        \item $(X,\mcB)$ and $(Y,\mcS)$ are Suslin measurable spaces;
        \item $T$ maps $Y$ onto $\mcF$, and 
        \item $\langle x,y \rangle \mapsto T(y)(x)$ is jointly measurable.
    \end{enumerate}

    Now consider $(\mcY_n,\mcA_n):= \prod_{i=1}^n(\mbR_i,\mcB_i)$, $\mcZ:=\bar{B}(\delta_0,\bm{\theta}_0) \times [0,1]$, and map $T:\mcZ \mapsto \mcH_n(\delta)$, $T(\bm{\theta},u)=h_{\bm{\theta},u}$. Both $(\mcY_n,\mcA_n)$ and $(\mcZ,\mcB(\mcZ))$ are complete separable metric spaces, hence both are Suslin spaces. For any $\By_n \in \mcY_n$, the evaluation map $\langle \By_n,(\bm{\theta},u) \rangle \mapsto h_{\bm{\theta},u}(\By_n)$ is jointly Borel measurable. Therefore $\mcH_n(\delta)$ is IAS via $(\mcZ,\mcB(\mcZ),T)$.

    Next, we show that $\mcH_n(\delta)$ is $P_{\bm{\theta}}^n$-measurable, i.e., for any $e_1,\dots,e_n \in \mbR$, 
    \begin{align*}
        (Y_1,\dots,Y_n) \mapsto \sup_{h_{\bm{\theta},u} \in \mcH_n(\delta)}\abs*{\sum_{i=1}^n e_i h_{i,\bm{\theta},u}(Y_i)}
    \end{align*}
    is measurable in the completion of $(\mcY_n,\mcA_n,P_{\bm{\theta}}^n)$. For any $t \ge 0$, and $e_1,\dots,e_n \in \mbR$, consider collection of sets:
    \begin{align*}
        E:=\left \{ \langle \By_n, (\bm{\theta},u) \rangle \in \mcY_n \times \mcZ: \abs*{\sum_{i=1}^n e_i h_{i,\bm{\theta},u}(y_i)} >t \right \}.
    \end{align*}
    As $\mcH_n(\delta)$ is IAS, and this is a finite sum, the map $\langle \By_n, (\bm{\theta},u) \rangle \mapsto \abs{\sum_{i=1}^n e_i h_{i,\bm{\theta},u}(y_i)}$ is jointly measurable. Therefore $E$ is product measurable, and thus Suslin. Hence its projection:
    \begin{align*}
        \left \{ \By_n \in \mcY_n: \sup_{h_{\bm{\theta},u} \in \mcH_n(\delta)}\abs*{\sum_{i=1}^n e_i h_{i,\bm{\theta},u}(y_i)} >t \right \}
    \end{align*}
    is also Suslin. Consequently, it's $P_{\bm{\theta}}^n$-measurable because all Suslin sets in a separable measurable space are universally measurable \citep{dudley2006course}. The same arguments also apply to $\mcH_n^2(\delta)$.
\end{proof}

\begin{proof}[Proof of Proposition \ref{prop:location_scale}]
    To bound $N(\epsilon,\mcH_n(\delta),\hat{d}_n)$ universally, it suffices to bound the number of unique length-$n$ binary vectors realized by $h_{\bm{\theta},u}$, as $(\bm{\theta},u)$ vary in $B(\delta,\bm{\theta}_0) \times [0,1]$, given any $I=\{(\Bx_1,y_1),\dots,(\Bx_n,y_n)\}$. Recall the coordinate functions $h_{i,\bm{\theta},u}(y_i) = 1(F(y_i;\Bx_i,\bm{\theta}) \le u)-1(F(y_i;\Bx_i,\bm{\theta}_0) \le u)$. Let $h_{i,\bm{\theta},u}^1(y_i)=1(F(y_i;\Bx_i,\bm{\theta}) \le u)$ be just the first term, and let $h_{\bm{\theta},u}^1$ denote the corresponding vector function. It is sufficient to show that the number of unique labels from $h_{\bm{\theta},u}^1$, denoted $N_1$, can be upper bounded by $K_1 n^{K_2}$ for some $K_1$, $K_2>0$. This is because the number of unique labels from the second term, denoted $N_2$ is obviously no larger than $N_1$, and the total number of unique labels from $h_{\bm{\theta},u}$ can be upper bounded by $N_1N_2 \le N_1^2$.

    Consider the following collection of sets
    \begin{align}\label{eq:S}
      \mcS = \{ \{(\Bx,y) \in \mcX \times \mbR: F(\Bx,y,\bm{\theta}) \le u\}, \bm{\theta} \in B(\delta,\bm{\theta}_0), u \in [0,1] \}.
    \end{align}   
    Given any finite set $I$ defined above, $N_1=\abs{\mcS \cap I}$, where $\mcS \cap I = \{ s \cap I: s \in \mcS \}$. And under the first scenario stated in Proposition \ref{prop:location_scale}, $\mcS$ is equivalent to
    \begin{align*}
        \mcS' := \{ \{(\Bx,y) \in \mcX \times \mbR: y-h(\Bx^\top \bm{\beta}) \le t\}, t \in \mbR,  \bm{\beta} \in B(\delta,\bm{\beta}_0) \}
    \end{align*}
    Fix any $t \in \mcR$, we start by studying $V(\mcS'_t)$, the VC dimension  of
    \begin{align*}
        \mcS'_t:= \{ \{(\Bx,y) \in \mcX \times \mbR: y-h(\Bx^\top \bm{\beta}) \le t\}, \bm{\beta} \in B(\delta,\bm{\beta}_0) \}.
    \end{align*}
    Without loss of generality, assuming $h$ is strictly monotone increasing, then $V(\mcS'_t)=V(\mcS'')$, where $\mcS'':=\{ \{(\Bx,z) \in \mcX \times \mbR: \Bx^\top \bm{\beta}-z \le 0\}, \bm{\beta} \in B(\delta,\bm{\beta}_0) \}$. To see why, suppose $\mcS'_t$ shatters $n$ points $\{(\Bx_1,y_1),\dots,(\Bx_n,y_n)\}$, then for any binary label $\Bl \in \{0,1\}^n$, there exists $\bm{\beta}_{\Bl}$ such that $l_i = 1(y_i-h(\Bx_i^\top \bm{\beta}_{\Bl}) \le t)$. Setting $z_i = h^{-1}(y_i-t)$, then obviously $\mcS''$ shatters $\{(\Bx_1,z_1), \dots, (\Bx_n,z_n)\}$, and vice versa.

    Notice $\mcS''$ is a collection of half-spaces in $\mbR^p \times \mbR$, therefore $V(\mcS'') \le p+1$. Hence by Sauer's lemma, there exists $K>0$ such that $\abs{\mcS'_t \cap I} \le K n^{p+1}$. Observe that for any one label realized in $\mcS'_t$, changing $t$ can at most create $n$ more labels because it is a thresholding parameter and it can flip at most 1 coordinate at a time. Therefore $N_1 \le K n^{p+1} \times (n+1) \le K' n^{p+2}$ for some $K'>0$.

    The proof for the second scenario stated in Proposition \ref{prop:location_scale} is similar, hence omitted.
    
\end{proof}

\begin{proof}[Proof of Proposition \ref{prop:o-minimal}]

    It suffices to show that for any such compact $A \subset \mbR$, $\mcS_A$ has a finite VC dimension $V$ constant in $n$. To see why, as discussed in the proof of Proposition \ref{prop:location_scale}, given any $I=\{(\Bx_1,y_1),\dots,(\Bx_n,y_n)\}$, in order to universally bound $N(\epsilon,\mcH_n(\delta),\hat{d}_n)$, it is sufficient to bound the number of unique length-$n$ binary vectors realized by $h_{\bm{\theta},u}^1$ as $(\bm{\theta},u)$ vary in $B(\delta,\bm{\theta}_0) \times [0,1]$, which is precisely $\abs{\mcS_A \cap I}$, for any compact $A$ such that $I \subset A^p \times A$. Therefore if $\mcS_A$ has VC dimension $V$, by Sauer's lemma, there exists $K>0$ such that $\abs{\mcS_A \cap I} \le K n^V$.

    We will show that any collection of sets is a VC class if it is a uniformly definable family in an o-minimal structure. We start by defining relevant terminologies. Let $L$ be a first order language consisting of $\{<,+,-,\cdot,0,1 \}$ and symbols of a collection of functions $\mcF$. Let $\mcM$ be an $L$-structure with domain $M$. Let $\psi(\Bz; \Bc)$ be a partitioned $L$-formula, with free variables $\Bz \in M^r$, and parameters $\Bc \in M^q$. We write $\mcM \models \psi(\Bz; \Bc)$ if $\psi(\Bz; \Bc)$ is true in $\mcM$. For any $B \subset M^r$, define $\psi(B;\Bc) = \{ \Bz \in B: \mcM \models \psi(\Bz;\Bc) \} \subset B$, then $\mcC_\psi = \{ \psi(M^r; \Bc): \Bc \in M^q \}$ is a \textit{uniformly definable family} in structure $\mcM$. If $M \subset \mbR$, a structure $\mcM$ is \textit{o-minimal} if for any $n \in \mbN$, every definable subsets of $\mbR^n$ has only finitely many connected components \citep{van2000field}. Theorem 1.2 of \cite{johnson2010compression} established that if $\mcM$ is o-minimal, then $\mcC_\psi$ has an extended $q$-compression. To make use of this theorem, we will further introduce results on sample compression schemes. 

    For any domain $X$, a \textit{concept} $c$ on $X$ is any subset of $X$. Let $C \subset 2^X$ be a concept class. For any $c \in C$, $x \in X$, $c(x) = 1(x \in c)$ gives the classification of $x$ based on concept $c$. Elements of $X \times \{0,1\}$ are called \textit{labeled examples}. For any finite set $B \subset X$, let $B^{\pm,c} \subset B \times \{0,1\}^{\abs{B}}$ denote the set $B$ of examples labeled by the concept $c$. A concept class $C$ on $X$ is said to have an \textit{extended $q$-compression scheme} if there exists a compression function $\kappa: \{S^{\pm,c}: S \subset X, \abs{S} \le \infty, c \in C\} \mapsto \{D^{\pm,c}: D \subset X, \abs{D} \le q, c \in C \}$, and a finite set $\mcR$ of reconstruction functions $\rho: \{D^{\pm,c}:D \subset X, \abs{D} \le q, c \in C \} \mapsto 2^X$, such that any finite $S \subset X$ and concept $c \in C$, $S_\kappa^{\pm,c}=\kappa(S^{\pm,c})$ is a set of at most $q$ labeled examples from $S^{\pm,c}$, and there exists a $\rho \in \mcR$ such that the hypothesis $h=\rho(S_\kappa^{\pm,c}) \subset X$ produces consistent labels as $S^{\pm,c}$, i.e., for all $s \in S, c(s)=h(s)$. See \cite{floyd1995sample} for details and examples. Now we present a lemma that connects extended $q$-compression schemes to VC dimension bounds.    

    \begin{lemma}\label{thm:compression}
        If a concept class $C$ on $X$ admits an extended $q$-compression scheme with $r=\abs{\mcR}$ reconstruction functions, its VC dimension can be upper bounded by a function of $q$ and $r$.        
    \end{lemma}

    \begin{proof}
        By the definition of an extended $q$-compression scheme, for any $S \subset X$ with $m$ elements, the number of unique length-$m$ binary labels created by concepts in $C$ is at most
        \begin{align*}
        \sum_{i=0}^q \binom{m}{i} r 2^i.
        \end{align*}
        This is because, each $S^{\pm,c}$ corresponds to one of the $r$ reconstruction functions, a compressed subset of at most $q$ elements, and their labels. 
        
        For $S$ to be shattered, the number of unique labels has to be greater than $2^m$. Therefore the VC dimension of $C$, denoted as $V$, can be upper bounded as:
        \begin{align*}
        2^V \le \sum_{i=0}^q \binom{V}{i} r 2^i \le r2^q \sum_{i=0}^q \binom{V}{i} \le r 2^q \cbr{\frac{Ve}{q}}^q.
        \end{align*} 
        Taking log on both sides, we have: $V \log2 \le \log r + q \log2 +q \log V + q - q \log q$. So we need to solve for: $V \le a + b \log V$ with $a=q + (q-q\log q + \log r)/\log 2$, and $b = q/\log2$. Manipulating the terms, we have:
        \begin{align*}
        \frac Vb &\le \frac ab + \log \cbr{\frac Vb} + \log b\\
        \frac 12 \frac Vb &\le \frac Vb - \log \cbr{\frac Vb} \le \frac ab + \log b \\
        V &\le 2(a+b \log b),
        \end{align*}
        which proves the claim as both $a$ and $b$ are functions of $q$ and $r$.
    \end{proof}

    Connecting all the dots, consider a first order language $L$ with symbols of the CDF $F$. Let $\mcM$ be an $L$-structure with domain $A$. Then $\mcS_{A'}:=\{ \{(\Bx,y) \in A^p \times A: F(\Bx,y,\bm{\theta}) \le u\}, \bm{\theta} \in A^k, u \in A \}$ can be written as $\mcC_\psi$, with $\psi(\Bz;\Bc)$ defined as $F(\Bx,y,\bm{\theta})\le u$, for $\Bz \in A^{p+1}$ consisting of $\Bx$ and $y$, and $\Bc \in A^{k+1}$ consisting of $\bm{\theta}$ and $u$. If $\mcM$ is a o-minimal structure, $\mcS_{A'}$ would be a uniformly definable family in $\mcM$: It admits an extended $q$-compression scheme, and hence its VC dimension can be upper bounded by a constant independent of $n$. Consequently, $\mcS_A \subset \mcS_{A'}$ is also a VC class and \ref{assmp:entropy} holds.

\end{proof}

\section{Proof of Theorem \ref{thm:CKS}}\label{app:CKS}

\begin{proof}

  Let $T_n$ denote the CKS test statistic. By Lemmas \ref{prop:paper5thm3.2}-\ref{lem:l2l3}, it suffices to show that equation \eqref{eq:thm_iid_c_step1} holds for $T_n$. Observe that $T_n = g \circ Z_n$ is the composition of two functions. The first, $Z_n:\mbR^n \to \ell^{\infty}(\mbR)$ is defined by $Z_n(\BY_n) \in \ell^{\infty}(\mbR)$, $Z_n(\BY_n)(y) = \sqrt{n}(P_n(y)-P_{\hat{\bm{\theta}}}(y))$ for $y \in \mbR$. The second, $g: \ell^{\infty}(\mbR) \to \mbR$, is the supremum norm functional, given by $g(h) =\norm{h}_\infty = \sup_{y \in \mbR} \abs{h(y)}$. Therefore it is sufficient to show $Z_n(\BY_n)$ converges in distribution under $P_{\bm{\theta}_n}$ to some random element $Z \in (\ell^\infty(\mbR), \mcB(\norm{\cdot}_\infty))$, and then apply the continuous mapping theorem. To simplify notation, we denote $Z_n(\BY_n)$ by $Z_n$ from now on. And we proceed in two steps:
  \begin{enumerate}[label=(\roman*)]
    \item First, we show weak convergence of the marginals, i.e., for each $y \in \mbR$, $Z_n(y) \rightsquigarrow Z(y)$ under $P_{\bm{\theta}_n}$;
    \item Second, we show that $Z_n$ is asymptotically tight under $P_{{\bm{\theta}}_n}$.
  \end{enumerate}

  We begin with step 1. By Taylor expansion, for any $y \in \mbR$, 
  \begin{align*}
    Z_n(y) = \sqrt{n}(P_n(y)-P_{\bm{\theta}_0}(y)) - \sqrt{n}\dot{P}_{\tilde{\bm{\theta}}_n}(y)(\hat{\bm{\theta}}_n - \bm{\theta}_0),
  \end{align*}
    for some $\tilde{\bm{\theta}}_n = t\bm{\theta}_0 + (1-t)\hat{\bm{\theta}}_n, \ t \in [0,1]$. By \ref{assmp:estimator}, this is equal to
  \begin{align}
    \underbrace{\sqrt{n}\cbr{\frac 1n \sum_{i=1}^n 1(Y_i \le y) - P_{\bm{\theta}_0}(y)} - \sqrt{n}\frac 1n \dot{P}_{\tilde{\bm{\theta}}_n}(y)I(\bm{\theta}_0)^{-1}\sum_{i=1}^n s_i(\bm{\theta}_0)}_{\tilde{Z}_n(y)} - \dot{P}_{\tilde{\bm{\theta}}_n}(y)o_p(\bm{1}_k)\label{eq:Zn}
  \end{align}
  By the Law of Large Numbers and \ref{assmp:estimator}, $\hat{\bm{\theta}}_n$ converges in probability to $\bm{\theta}_0$. Consequently, the same holds for $\tilde{\bm{\theta}}_n$. By \ref{assmp:diff_A}, $\dot{P}_{\bm{\theta}}$ is continuous in ${\bm{\theta}}$, hence $\dot{P}_{\tilde{\bm{\theta}}_n}(y)$ converges in probability to $\dot{P}_{\bm{\theta}_0}(y)$. Therefore the last term in \eqref{eq:Zn} is $o_p(1_k)$, and we only need to consider $\tilde{Z}_n(y)$ to study the asymptotic distribution of $Z_n(y)$. By the Central Limit Theorem,
  \begin{align*}
    \sqrt{n} \frac 1n \sum_{i=1}^n \begin{pmatrix}
        s_i(\bm{\theta}_0) \\
        1(Y_i \le y) - P_{\bm{\theta}_0}(y)
    \end{pmatrix} \rightsquigarrow 
    N\cbr{0, \Sigma}, \quad \Sigma = \begin{pmatrix}
        I(\bm{\theta}_0) & \bm{\tau} \\
        \bm{\tau}^\top & P_{\bm{\theta}_0}(y)(1-P_{\bm{\theta}_0}(y))
    \end{pmatrix},
  \end{align*}
  where $\bm{\tau} \in \mbR^k$ is the asymptotic covariance. Therefore $\tilde{Z}_n(y)$ is asymptotically normal by Slutsky's theorem.

  By Lemma 7.6 of \cite{van2000asymptotic}, \ref{assmp:diff_A} implies that this family of sampling model is QMD at $\bm{\theta}_0$. Hence for all $\Bh_n \to \Bh \in \mbR^k$, $\bm{\theta}_n = \bm{\theta}_0 + \Bh_n/\sqrt{n}$, the log-likelihood ratio $\Lambda_n$ admits the following expansion:
  \begin{align}\label{eq:llhr}
      \Lambda_n = \log \frac{dP_{{\bm{\theta}}_n}^n}{dP_{\bm{\theta}_0}^n} = \frac 1{\sqrt{n}} \sum_{i=1}^n \Bh^\top s_i(\bm{\theta}_0) - \frac 12 \Bh^\top I({\bm{\theta}_0})\Bh + o_p(1).
  \end{align}
  Therefore $\tilde{Z}_n(y)$ and $\Lambda_n$ are jointly asymptotically normal. By Le Cam's third lemma, if the asymptotic covariance between them is zero, we can conclude that $\tilde{Z}_n(y) \rightsquigarrow Z(y)$ under $P_{{\bm{\theta}}_n}$. By equation \eqref{eq:llhr}, it is sufficient to show that the asymptotic covariance between $\tilde{Z}_n(y)$ and $\BU_n=1/\sqrt{n} \sum_{i=1}^n s_i(\bm{\theta}_0)$ is zero. Let
  \begin{align*}
      V_n = \frac 1{\sqrt{n}} \sum_{i=1}^n (1(Y_i \le y) - P_{\bm{\theta}_0}(y)), \quad \BB_n = \begin{pmatrix}
          -\dot{P}_{\tilde{\bm{\theta}}_n}(y)I(\bm{\theta}_0)^{-1} & 1 \\
          \BI_p & \bm{0}
      \end{pmatrix} \in \mbR^{(p+1) \times (p+1)}.
  \end{align*}
  We have already shown:
  \begin{align*}
      \begin{pmatrix}
          \BU_n \\
          V_n
      \end{pmatrix} \rightsquigarrow N(\bm{0},\Sigma), \quad \BB_n \overset{p}{\to} \BB = \begin{pmatrix}
          -\dot{P}_{\bm{\theta}_0}(y)I(\bm{\theta}_0)^{-1} & 1 \\
          \BI_p & \bm{0}
      \end{pmatrix}.
  \end{align*}
  Therefore
  \begin{align*}
      \begin{pmatrix}
          \tilde{Z}_n(y) \\
          \BU_n
      \end{pmatrix} = \BB_n \begin{pmatrix}
          \BU_n \\
          V_n
      \end{pmatrix} \rightsquigarrow N(\bm{0}, \BB \Sigma \BB^\top).
  \end{align*}
  It can be shown that the covariance term of $\BB \Sigma \BB^\top$ is $\bm{\tau}^\top-\dot{P}_{\bm{\theta}_0}(y)=0$ because by \ref{assmp:domination_A},
  \begin{align*}
    \bm{\tau} = \mathrm{cov}(1(Y_i \le y), s_i(\bm{\theta}_0)) = \frac{\partial}{\partial {\bm{\theta}}} \int p_{\bm{\theta}}(u)1(u \le y) du \Big|_{{\bm{\theta}}=\bm{\theta}_0}= \dot{P}_{\bm{\theta}_0}(y)^\top.
  \end{align*}

  We now prove step 2. Under assumptions \ref{assmp:diff_A} and \ref{assmp:estimator}, $Z_n$ converges weakly to some tight and measurable random element $Z \in \ell^\infty(\mbR)$ under $P_{\bm{\theta}_0}$ (Theorem 19.23, \cite{van2000asymptotic}), therefore $Z_n$ is asymptotically tight under $P_{\bm{\theta}_0}$. The log-likelihood ratio $\Lambda_n$'s are uniformly tight under $P_{\bm{\theta}_0}$ because
  \begin{align*}
      E[\Lambda_n] = E\sbr{\log \frac{dP_{{\bm{\theta}}_n}^n}{dP_{\bm{\theta}_0}^n}} \le \log \cbr{E\sbr{\frac{dP_{{\bm{\theta}}_n}^n}{dP_{\bm{\theta}_0}^n}}} \le \log(1) = 0.
  \end{align*}
  Therefore $Z_n$ and $\Lambda_n$ are jointly asymptotically tight under $P_{\bm{\theta}_0}$, hence
  \begin{align*}
      (Z_n, \Lambda_n) \rightsquigarrow (Z,V) \in (\ell^\infty(\mbR) \times \mbR),
  \end{align*}
  where $V$ is some mean-zero normal random variable. By Le Cam's third lemma, $Z_n$ converges weakly under $P_{{\bm{\theta}}_n}$ to a tight limit process, therefore $Z_n$ is asymptotically tight under $P_{{\bm{\theta}}_n}$. 

  This completes the proof.

\end{proof}

\section{Proof of Theorem \ref{thm:GKS}}\label{app:GKS}

Following the decomposition in equation \eqref{eq:GKS_decompose}, we first show term (c)$=o_p(1)$ uniformly in $u$, which is equivalent to showing $\sup_{u \in [0,1]} \abs{\mbG_n(h_{\hat{\bm{\theta}}_n,u})} \overset{p^*}{\to} 0$. For all $\epsilon>0$, 
\begin{align*}
    P^*\cbr{\sup_{u \in [0,1]}\abs{\mbG_n(h_{\hat{\bm{\theta}}_n,u})} > \epsilon} \le P(\norm{\hat{\bm{\theta}}_n-\bm{\theta}_0}>\delta_n) + P^*(\norm{\mbG_n(h_{\bm{\theta},u})}_{\mcH(\delta_n)}>\epsilon).
\end{align*}
By \ref{assmp:estimator}, $\norm{\hat{\bm{\theta}}_n-\bm{\theta}_0}=O_p(1/\sqrt{n})$, hence $P(\norm{\hat{\bm{\theta}}_n-\bm{\theta}_0}>\delta_n) \to 0$ for $\delta_n=O(\log(n)/\sqrt{n})$. Therefore it suffices to show $E^*[\norm{\mbG_n(h_{\bm{\theta},u})}_{\mcH(\delta_n)}] \to 0$, such that the second term also converges to zero by Markov's inequality. 

Let $e_1,\dots, e_n$ be i.i.d. Rademacher random variables independent of the $Y_i$'s. Define $\mbG_n^{o}(h_{\bm{\theta},u}):= 1/\sqrt{n} \sum_{i=1}^n e_i h_{i,\bm{\theta},u}(Y_i)$. Then by symmetrization, 
\begin{align}\label{eq:symmetrization}
    E^*[\norm{\mbG_n(h_{\bm{\theta},u})}_{\mcH(\delta_n)}] \le 2 \bar{E}_{e,Y}[\norm{\mbG_n^o(h_{\bm{\theta},u})}_{\mcH(\delta_n)}].
\end{align}
We use $\bar{E}$ on the right hand side because by Lemma \ref{lem:ias}, $\norm{\mbG_n^o(h_{\bm{\theta},u})}_{\mcH(\delta_n)}$ is measurable in the completion of $P_{\bm{\theta}_0}$. We highlight that $Y_i$'s are independent but not identically distributed, and we are working with coordinate-wise different vector functions $h_{\bm{\theta},u}$, the symmetrization results we are using here are slightly different to the classical results for i.i.d. data (see Lemma 2.3.1 in \cite{van2023weak}). The proof is very similar though, hence omitted.

Now $\{\mbG_n^o(h_{\bm{\theta},u}): h_{\bm{\theta},u} \in \mcH(\delta_n) \}$ is a separable sub-Gaussian process, hence we can apply Dudley's entropy integral bound (see Corollary 2.2.9 in \cite{van2023weak}), i.e., there exists $K_3>0$ such that
\begin{align}\label{eq:dudley}
    \bar{E}_{e,Y}[\norm{\mbG_n^o(h_{\bm{\theta},u})}_{\mcH(\delta_n)}] &\le \bar{E}_{e,Y}\abs{\mbG_n(h_{\bm{\theta}_0,u_0})} + K_3 \bar{E}_Y \sbr{\int_0^{\hat{\sigma}_n(\delta_n)} \sqrt{\log N(\epsilon,\mcH(\delta_n), \hat{d}_n)} d\epsilon} \nonumber \\
    &\le K_3 \sqrt{\log(K_1) + K_2 \log(n)} \bar{E}[\hat{\sigma}_n(\delta_n)],
\end{align}
where $\hat{\sigma}_n(\delta_n)$ is the diameter of $\mcH(\delta_n)$ under $\hat{d}_n$. For notational simplicity, we will write $\hat{\sigma}(\delta_n)$ henceforth. The second inequality holds by \ref{assmp:entropy}, and that $\mbG_n(h_{\bm{\theta}_0,u_0})\equiv 0$ for any $u_0$, by the definition of $h_{\bm{\theta},u}$. Let $\norm{\cdot}_{L_2(\hat{P}_n)}$ denote the empirical $L_2$ norm, $\norm{\cdot}_{L_2(P_n)}$ denote the population $L_2$ norm. Observe 
\begin{align}\label{eq:diam}
    &\bar{E}[\hat{\sigma}(\delta_n)] \le \sqrt{\bar{E}[\hat{\sigma}(\delta_n)^2]} \le 2\sqrt{\bar{E}\sbr{\sup_{h_{\bm{\theta},u} \in \mcH(\delta_n)} \norm{h_{\bm{\theta},u}}_{L_2(\hat{P}_n)}^2}} \nonumber \\
    \le& 2\sqrt{\bar{E}\sbr{\abs*{\sup_{h_{\bm{\theta},u} \in \mcH(\delta_n)} \norm{h_{\bm{\theta},u}}_{L_2(\hat{P}_n)}^2 - \sup_{h_{\bm{\theta},u} \in \mcH(\delta_n)} \norm{h_{\bm{\theta},u}}_{L_2(P_n)}^2} + \sup_{h_{\bm{\theta},u} \in \mcH(\delta_n)} \norm{h_{\bm{\theta},u}}_{L_2(P_n)}^2}} \nonumber \\
    \le& 2\sqrt{\bar{E}\sbr{\frac 1{\sqrt{n}} \norm{\mbG_n(h_{\bm{\theta},u}^2)}_{\mcH^2(\delta_n)}} + \sup_{h_{\bm{\theta},u} \in \mcH(\delta_n)}\norm{h_{\bm{\theta},u}}_{L_2(P_n)}^2},
\end{align}
with $\mcH^2(\delta_n)$ defined in equation \eqref{eq:Hn2}. We will show both terms vanish asymptotically.

By \ref{assmp:diff_B} and the proof of Proposition \ref{prop:frechet}, there exists $M_2>0$ such that 
\begin{align*}
    \sup_{i \in \mbN} \sup_{\bm{\theta} \in B(\delta_0,\bm{\theta}_0)} \sup_{y \in \mbR} \norm{\dot{F}(y;\Bx_i,\bm{\theta})} \le M_2.
\end{align*}
Then notice that each $h_{i,\bm{\theta},u}(Y_i)^2$ is non-zero if and only if $u$ is in between $F(Y_i;\Bx_i,\bm{\theta})$ and $F(Y_i;\Bx_i,\bm{\theta}_0)$. Therefore 
\begin{align*}
    E[h_{i,\bm{\theta},u}(Y_i)^2] &\le P[\abs{F(Y_i;\Bx_i,\bm{\theta}_0)- u} \le \abs{F(Y_i;\Bx_i,\bm{\theta}_0)- F(Y_i;\Bx_i,\bm{\theta})}] \\
    &\le P[\abs{F(Y_i;\Bx_i,\bm{\theta}_0)- u} \le \norm{\dot{F}(Y_i;\Bx_i,\tilde{\bm{\theta}})} \norm{\bm{\theta}-\bm{\theta}_0}] \\
    &\le P[\abs{F(Y_i;\Bx_i,\bm{\theta}_0)- u} \le M_2\delta_n] \le 2M_2 \delta_n,
\end{align*}
where $\tilde{\bm{\theta}}$ is some convex combination of $\bm{\theta}$ and $\bm{\theta}_0$, and the last inequality holds because $F(Y_i;\Bx_i,\bm{\theta}_0)$ is uniformly distributed under $P_{\bm{\theta}_0}$. Therefore $\sup \norm{h_{\bm{\theta},u}}_{L_2(P_n)}^2 \le 2M_2\delta_n = O(\log(n)/\sqrt{n})$.

Next, we bound $\bar{E}[\norm{\mbG_n(h_{\bm{\theta},u}^2)}_{\mcH^2(\delta_n)}]$. By symmetrization and Dudley's entropy integral bound again, there exists $K_4>0$,
\begin{align}\label{eq:ULLN_bound}
    \bar{E}[\norm{\mbG_n(h_{\bm{\theta},u}^2)}_{\mcH^2(\delta_n)}] &\le 2 \bar{E}_{e,Y}[\norm{\mbG_n^o(h_{\bm{\theta},u}^2)}_{\mcH^2(\delta_n)}] \nonumber \\
    &\le K_4 \bar{E} \sbr{\int_0^\infty \sqrt{\log N(\epsilon,\mcH^2(\delta_n),\hat{d}_n})d\epsilon}.
\end{align}
Notice $\hat{d}_n(h_{\bm{\theta},u}^2, h_{\bm{\theta}',u'}^2) \le 2 \hat{d}_n(h_{\bm{\theta},u}, h_{\bm{\theta}',u'})$ because 
\begin{align*}
   \hat{d}_n(h_{\bm{\theta},u}^2, h_{\bm{\theta}',u'}^2)^2 &= \frac 1n \sum_{i=1}^n (h_{i,\bm{\theta},u}(Y_i)-h_{i,\bm{\theta}',u'}(Y_i))^2(h_{i,\bm{\theta},u}(Y_i)+h_{i,\bm{\theta}',u'}(Y_i))^2 \\
   &\le 4 \frac 1n \sum_{i=1}^n (h_{i,\bm{\theta},u}(Y_i)-h_{i,\bm{\theta}',u'}(Y_i))^2.
\end{align*}
Therefore any $(\epsilon/2)$-net of $\mcH(\delta_n)$ is an $\epsilon$-net of $\mcH^2(\delta_n)$ under $\hat{d}_n$, and equation \eqref{eq:ULLN_bound} can be upper bounded by 
\begin{align*}
    K_4 \bar{E} \sbr{\int_0^\infty \sqrt{\log N(\epsilon/2,\mcH(\delta_n),\hat{d}_n})d\epsilon} \le K_4 \sqrt{\log(K_1) + K_2 \log(n)} = O(\sqrt{\log(n)}).
\end{align*}
Putting these results together with equations \eqref{eq:symmetrization},\eqref{eq:dudley} and \eqref{eq:diam},
\begin{align*}
    E^*[\norm{\mbG_n(h_{\bm{\theta},u})}_{\mcH(\delta_n)}] &\le O(\sqrt{\log(n)})\sqrt{O(\log(n)/\sqrt{n})+O(\sqrt{\log(n)}/\sqrt{n})} \\
    &= O(\log(n)/n^{1/4}) \to 0,
\end{align*}
which completes the proof for step 1, and shows that the asymptotic distribution of $\tilde{Z}_n$ is only dependent on term (a)+(b).

Next, we use Le Cam's third lemma to show that for each $u$, (a)+(b) satisfies equation \eqref{eq:thm_iid_c_step1}, and it is asymptotically tight. We start with term (b), observe:
\begin{align*}
    m_i(\bm{\theta}_0,u)-m_i(\hat{\bm{\theta}}_n,u)&=F(Q(u;\Bx_i,\bm{\theta}_0);\Bx_i,\bm{\theta}_0)-F(Q(u;\Bx_i,\hat{\bm{\theta}}_n);\Bx_i,\bm{\theta}_0) \\
    &=F(Q(u;\Bx_i,\hat{\bm{\theta}}_n);\Bx_i,\hat{\bm{\theta}}_n)-F(Q(u;\Bx_i,\hat{\bm{\theta}}_n);\Bx_i,\bm{\theta}_0) \\
    &= \dot{F}(Q(u;\Bx_i,\hat{\bm{\theta}}_n);\Bx_i,\tilde{\bm{\theta}}_{ni})^\top (\hat{\bm{\theta}}_n-\bm{\theta}_0) \\
    &= \dot{F}(Q(\tilde{u}_{ni};\Bx_i,\tilde{\bm{\theta}}_{ni});\Bx_i,\tilde{\bm{\theta}}_{ni})^\top (\hat{\bm{\theta}}_n-\bm{\theta}_0) \\
    &= g_i(\tilde{u}_{ni},\tilde{\bm{\theta}}_{ni})^\top (\hat{\bm{\theta}}_n-\bm{\theta}_0),
\end{align*}
where $\tilde{\bm{\theta}}_{ni}=t\bm{\theta}_0+(1-t)\hat{\bm{\theta}}_n$ for some $t \in [0,1]$, and $\tilde{u}_{ni}=F(Q(u;\Bx_i,\hat{\bm{\theta}}_n);\Bx_i,\tilde{\bm{\theta}}_{ni})$. By \ref{assmp:diff_B} and the proof of Proposition \ref{prop:frechet}, $\tilde{u}_{ni} \overset{p}{\to} u$ uniformly in $u$ and $i$ as $\hat{\bm{\theta}}_n \overset{p}{\to} \bm{\theta}_0$. And because $g_i$ is uniformly continuous in $u$ and $\bm{\theta}$, $g_i(\tilde{u}_{ni},\tilde{\bm{\theta}}_{ni}) \overset{p}{\to} g_i(u,\bm{\theta}_0)$ uniformly in $u$ and $i$. Let $\BU_n:=1/\sqrt{n}\sum_{i=1}^n s_i(\bm{\theta}_0)$, by \ref{assmp:estimator}, term (b) can be decomposed as
\begin{align*}
    &\cbr{\frac 1n \sum_{i=1}^n [g_i(\tilde{u}_{ni},\tilde{\bm{\theta}}_{ni})+r_{ni}(u)]}^\top (\sqrt{n}(\hat{\bm{\theta}}_n-\bm{\theta}_0)) \\
    =&\cbr{\frac 1n \sum_{i=1}^n g_i(u,\bm{\theta}_0)}^\top \BI(\bm{\theta}_0)^{-1} \BU_n + \cbr{\frac 1n \sum_{i=1}^n g_i(u,\bm{\theta}_0)}^\top o_p(\bm{1}_k) + \\
    &\cbr{\frac 1n \sum_{i=1}^n r_{ni}(u)}^\top(\sqrt{n}(\hat{\bm{\theta}}_n-\bm{\theta}_0)).
\end{align*}
Because $r_{ni}(u)=o_p(1)$ uniformly in $u$ and $i$, the last term is $o_p(1)$ uniformly in $u$. By \ref{assmp:diff_B}, $1/n\sum_{i=1}^n g_i(u,\bm{\theta}_0) \to g(u)$ uniformly in $u$, therefore the second term is $o_p(1)$ uniformly in $u$, and the first term is asymptotically tight. Consequently, the asymptotic distribution of $\tilde{Z}_n(u)$ depends on
\begin{align*}
    \tilde{Z}_n^*(u)=\frac 1{\sqrt{n}}\sum_{i=1}^n \cbr{u - 1(F(Y_i;\Bx_i,\bm{\theta}_0) \le u)} + \cbr{\frac 1n \sum_{i=1}^n g_i(u,\bm{\theta}_0)}^\top \BI(\bm{\theta}_0)^{-1} \BU_n.
\end{align*}
The first term converges in distribution to a Brownian bridge, hence $\tilde{Z}_n^*$ is asymptotically tight. Now to use Le Cam's third lemma, it suffices to show that for each $u$, $\tilde{Z}_n^*(u)$ and the log-likelihood ratio $\Lambda_n$ are jointly asymptotically normal with zero asymptotic covariance.

\ref{assmp:supp_B}, \ref{assmp:diff_B}, \ref{assmp:Fisher_B} and \ref{assmp:ui} are sufficient to invoke Theorem 3.1 in \cite{philippou1973asymptotic} to show that the sequence of models $\{P_{\bm{\theta}}^n: \bm{\theta} \in \Theta\}$ is locally asymptotically normal: for all $\Bh_n \to \Bh \in \mbR^k$, $\bm{\theta}_n=\bm{\theta}_0+\Bh_n/\sqrt{n}$, the log-likelihood ratio $\Lambda_n$ admits the same asymptotic expansion as in equation \eqref{eq:llhr}. From here, it is obvious that for each $u$, $\Lambda_n$ and $\tilde{Z}_n^*(u)$ are jointly asymptotically normal. And to study their asymptotic covariance, it suffices to look at the asymptotic covariance between $\BU_n$ and each term of $\tilde{Z}_n^*(u)$. For the first term:
\begin{align*}
  &\lim_{n \to \infty} \frac 1n \sum_{i=1}^n Cov(u-1(F(Y_i;\Bx_i,\bm{\theta}_0) \le u), s_i(\bm{\theta}_0)) \\
  =&\lim_{n \to \infty} \frac 1n \sum_{i=1}^n -E\sbr{1(F(Y_i;\Bx_i,\bm{\theta}_0) \le u) s_i(\bm{\theta}_0)} \\
  =& \lim_{n \to \infty} \frac 1n \sum_{i=1}^n - \frac{\partial}{\partial \bm{\theta}} \int 1(y \le Q(u;\Bx_i,\bm{\theta}))dy \Bigg|_{\bm{\theta}=\bm{\theta}_0}\\
  =& \lim_{n \to \infty} \frac 1n \sum_{i=1}^n -g_i(u,\bm{\theta}_0) = -g(u),
\end{align*}
where the second to last line holds by \ref{assmp:domination_B}, and the last inequality holds by \ref{assmp:diff_B}. For the second term:
\begin{align*}
  &\lim_{n \to \infty} \frac 1n Cov \cbr{\sum_{i=1}^n g_i(u,\bm{\theta}_0)^\top \BI(\bm{\theta}_0)^{-1} \BU_n, \BU_n} \\
  =&\lim_{n \to \infty} \cbr{\frac 1n \sum_{i=1}^n g_i(u,\bm{\theta}_0)} \BI(\bm{\theta}_0)^{-1} \cbr{\frac 1n \sum_{i=1}^n E[s_i(\bm{\theta}_0)s_i(\bm{\theta}_0)^\top]}=g(u).
\end{align*}
Therefore the overall asymptotic covariance is zero. This completes the proof.

\section{Proof of Corollary \ref{coro:regression}}\label{app:coro}

It is obvious that Proposition \ref{prop:location_scale} applies to location-scale family models such as the normal linear regression, Student-t regression, Gamma GLM (for both log and inverse link functions), and Weibull GLM. Lognormal regression model is a location-scale family on the log scale, hence it is also covered by Proposition \ref{prop:location_scale}. We will show Proposition \ref{prop:o-minimal} is more general. It applies not only to the models listed above, but also to other regression models such as the Beta GLM and inverse Gaussian GLM.

Following the proof for Proposition \ref{prop:o-minimal}, it suffices to show that the CDF $F$ consists of functions that are definable in an o-minimal structure under any compact domain $A \subset \mbR$. If this holds, $\mcS_A$ defined in Proposition \ref{prop:o-minimal} would be a subset of a uniformly definable family in an o-minimal structure. 

\cite{speissegger1997pfaffian} presented one of the largest o-minimal structures, denoted as $Pf(\mbR_{\mcG})$, where the following functions are definable: polynomials, restricted analytical functions, exponential function, Gamma function and erf function. For all the regression models listed in Corollary \ref{coro:regression}, the CDF $F$ possibly consists of polynomials, exponential functions, Gamma function, erf function, lower incomplete Gamma function, incomplete Beta function, and hypergeometric function. All of these functions are definable in $Pf(\mbR_{\mcG})$ on a compact domain $A \subset \mbR$. Therefore Proposition \ref{prop:location_scale} holds for all of these regression models.

\section{Proof for numerical experiments setup}\label{app:sim}

\subsection{Gamma model}\label{app:sim_gamma}

We will show that the setup for the Gamma model numerical experiment in Section \ref{sec:sim_gamma} satisfy \ref{assmp:supp_A}-\ref{assmp:estimator}. 
 
\begin{proof}
Let $P_{\alpha,\beta}(y)$ denote the CDF of $\mathrm{Gamma}(\alpha,\beta)$.

\begin{enumerate}[label=A\arabic{*}.]
  \item Holds obviously. 
  \item It is obvious that $P_{\alpha,\beta}(y)$ is continuously differentiable in both $(\alpha,\beta)$ and $y$. We will show $P_{\alpha,\beta}(y)$ satisfies conditions in Proposition \ref{prop:frechet}. Observe $P_{\alpha,\beta}(y) = \gamma(\alpha,\beta y)/\Gamma(\alpha)$, where $\gamma(\alpha,\beta y)= \int_0^{\beta y} u^{\alpha-1} e^{-u}du$ is the lower incomplete gamma function. We compute $\dot{P}_{\alpha,\beta}(y)$, the gradient of $P_{\alpha,\beta}(y)$ with respect to $(\alpha,\beta)$.
  \begin{align*}
      \dot{P}_{\alpha,\beta}(y) = \begin{pmatrix}
          \frac{\partial P_{\alpha,\beta}(y)}{\partial \alpha} \\
          \frac{\partial P_{\alpha,\beta}(y)}{\partial \beta}
      \end{pmatrix} = \begin{pmatrix}
          -\frac{\psi(\alpha)}{\Gamma(\alpha)} \gamma(\alpha, \beta y) + \frac 1{\Gamma(\alpha)}\frac{\partial \gamma}{\partial \alpha} \\
          \frac 1{\Gamma(\alpha)}\frac{\partial \gamma}{\partial \beta}
      \end{pmatrix}
  \end{align*}
  where
  \begin{align*}
    \frac{\partial \gamma}{\partial \alpha} = \int_{0}^{\beta y} u^{\alpha-1}\log u e^{-u}du, \quad \frac{\partial \gamma}{\partial \beta} = \beta^{\alpha-1}y^\alpha e^{-\beta y}.
  \end{align*}
  Because $P_{\alpha,\beta}(y)$ is jointly continuous in $(\alpha,\beta,y)$ and it's strictly increasing in $y$, its quantile function $Q(u,\alpha,\beta)$ is continuous in $(\alpha,\beta,u)$. $\dot{P}_{\alpha,\beta}(y)$ is obviously continuous in $(\alpha,\beta,y)$, hence the function $g(u,\alpha,\beta)=\dot{P}_{\alpha,\beta}(Q(u,\alpha,\beta))$ defined in Proposition \ref{prop:frechet} is continuous in $(u,\alpha,\beta)$ on $(0,1) \times \bar{B}(\delta_0,(\alpha_0,\beta_0))$. It can be easily shown that $\lim_{u \to 0^+}g(u,\alpha,\beta) = \bm{0}$, $\lim_{u \to 1^-}g(u,\alpha,\beta) = \bm{0}$. This is sufficient for the condition of Proposition \ref{prop:frechet}, hence $P_{\alpha,\beta}$ is Fréchet differentiable.
  \item To show \ref{assmp:domination_A} holds, we will show that the sufficient condition discussed in Remark \ref{rmk:domination} is satisfied. The score function can be bounded element wise:
  \begin{align*}
    \abs{s(\alpha,\beta)} = \begin{pmatrix}
      \abs{\log\beta - \psi(\alpha) + \log Y} \\
      \abs{\alpha/\beta - Y}
    \end{pmatrix} \le \begin{pmatrix}
      \log\beta + \psi(\alpha) + \log Y \\
      \alpha/\beta + Y
    \end{pmatrix},
  \end{align*}
  and both $\log Y$ and $Y$ are integrable. Therefore \ref{assmp:domination_A} holds.
  \item Conditions on the Fisher information matrix clearly hold for the Gamma distribution. 
  \item Under the Gamma distribution, MLEs for $\alpha$ and $\beta$ are CAN estimators. The priors we considered are all continuous with positive density on the entire parameter space. Also the parameters of a Gamma distribution are identifiable such that for all $\epsilon>0$, 
  \begin{align*}
    \inf_{\norm{(\alpha,\beta)-(\alpha',\beta')} > \epsilon} \sup_{y \in \mbR^+}\abs{P_{\alpha,\beta}(y)-P_{\alpha',\beta'}(y)}>0.
  \end{align*}
  Therefore Lemma 10.4, and hence Theorem 10.1 the BvM theorem in \cite{van2000asymptotic} holds. 
  \item Both $\log Y$ and $Y$ are square integrable. Therefore as discussed in Remark \ref{rmk:estimator}, this is sufficient to invoke Theorem 5.39 in \cite{van2000asymptotic} to conclude that the MLEs admit the expansion stated in \ref{assmp:estimator}. By \ref{assmp:BvM}, the same expansion applies to the posterior means.
\end{enumerate}
\end{proof}

\subsection{Gamma GLM}\label{app:sim_gglm}

We assume $\Bx_i \in \mcX$ for some compact $\mcX \subset \mbR^p$, and $\sum_{i=1}^n \Bx_i \Bx_i^\top/n \to \BQ$ for some positive define matrix $\BQ \in \mbR^{p \times p}$. Under these conditions, we will show that the setup for the Gamma GLM numerical experiment in Section \ref{sec:sim_gglm} satisfy \ref{assmp:BvM}-\ref{assmp:entropy}.

\begin{proof}

\ref{assmp:supp_B} holds trivially. \ref{assmp:entropy} holds by Corollary \ref{coro:regression}. Under the assumed regularity conditions on the covariates, the proof strategy for the Gamma model example also applies to \ref{assmp:BvM}, \ref{assmp:diff_B}-\ref{assmp:Fisher_B}. 

The assumed regularity condition on the covariates satisfy assumptions for Theorem 5 of \cite{fahrmeir1985consistency}, which ensures that there exists a sequence of MLE $\hat{\bm{\theta}}_n \to \bm{\theta}_0$ a.s.. Therefore Theorem 2.2 in \cite{he1996general} holds, and the MLE admits the expansion in \ref{assmp:estimator}. The same expansion also applies to the posterior mean by \ref{assmp:BvM}.

Finally, we show \ref{assmp:ui} holds for $\gamma=1$. The score function at the true parameter values can be derived as:
\begin{align*}
    s_i(\alpha,\bm{\beta}) = \begin{pmatrix}
        \log(\alpha_0)+1-\Bx_i^\top \bm{\beta}_0-\psi(\alpha_0) + \log(Y_i)-Y_i \exp(-\Bx_i^\top \bm{\beta}_0) \\
        \alpha_0[Y_i\exp(-\Bx_i^\top \bm{\beta}_0)-1]\Bx_i
    \end{pmatrix}.
\end{align*}
$\sup_{i \in \mbN} E\sbr{\norm{s_i(\alpha,\bm{\beta})}^3} < \infty$ because $\Bx_i$ lives in a compact space $\mcX$, and all the moments of $Y_i$ exists under the Gamma GLM.

\end{proof}

\end{appendix}

\end{document}